\documentclass{amsart}
\usepackage{amsmath,amscd,xypic,amssymb,combelow,color,enumitem,graphicx,float,comment,mathtools}
\usepackage{tikz-cd}
\usepackage[hidelinks]{hyperref}
\definecolor{citation}{rgb}{0,.40,.80}
\hypersetup{
   colorlinks=true,
   citecolor=black,
    linkcolor=citation,
    }
\usepackage[style=alphabetic, backend=biber, eprint=true, maxbibnames=99, doi=false, url = false, isbn=false]{biblatex}
\usepackage[T2A]{fontenc}
\usepackage[russian,english]{babel}
\usepackage{quiver}
\usetikzlibrary{babel}
\newtheorem{theorem}{Theorem}[section]

\newtheorem{proposition}[theorem]{Proposition}
\newtheorem{lemma}[theorem]{Lemma}

\newtheorem{corollary}[theorem]{Corollary}

\newtheorem{definition}[theorem]{Definition}
\newtheorem{claim}[theorem]{Claim}
\newtheorem{example}[theorem]{Example}
\newtheorem{remark}[theorem]{Remark}

\def\fg{{\mathfrak{g}}}

\def\fn{{\mathfrak{n}}}

\def\GL{\mathrm{GL}}

\def\BC{{\mathbb{C}}}

\def\BN{{\mathbb{N}}}
\def\BF{{\mathbb{F}}}

\def\BQ{{\mathbb{Q}}}
\def\BZ{{\mathbb{Z}}}

\DeclareMathOperator{\Tr}{\mathrm{Tr}}
\def\vir{\mathrm{vir}}

\DeclareMathOperator{\ptau}{{}^{\mathfrak{p}}\tau}
\DeclareMathOperator{\JH}{\mathrm{JH}}
\DeclareMathOperator{\Per}{\mathrm{Perv}}

\def\oi{\overline{i}}

\def\CA{{\mathcal{A}}}

\def\DD{{\mathcal{D}}}

\def\CH{{\mathcal{H}}}

\def\CM{{\mathcal{M}}}

\def\CS{{\mathcal{S}}}

\def\CV{{\mathcal{V}}}

\def\sym{\textrm{sym}}
\def\Sym{\textrm{Sym}}

\def\bd{{\mathbf{d}}}

\def\bk{{\mathbf{k}}}

\def\br{{\mathbf{r}}}

\def\bs{\boldsymbol{\varsigma}}
\def\bell{\boldsymbol{\ell}}

\def\nn{{\mathbb{N}^I}}
\def\zz{{\mathbb{Z}^I}}

\def\bm{{\boldsymbol{m}}}
\def\bn{{\boldsymbol{n}}}

\def\b0{{\boldsymbol{0}}}

\def\loccit{\emph{loc.~cit.~}}

\def\hdeg{\text{hdeg }}

\def\Spec{\text{Spec}}

\def\Sym{\text{Sym}}

\def\edge{\Omega}
\def\dedge{\bar{\edge}}
\def\dalpha{\bar{\alpha}}
\def\dQ{\bar{Q}}
\def\tQ{\tilde{Q}}

\def\tW{\tilde{W}}

\def\ring{R}
\def\field{\BF}

\def\ozeta{\bar{\zeta}}

\def\soft{\begin{otherlanguage*}{russian}Ь \end{otherlanguage*}}
\def\hard{\begin{otherlanguage*}{russian}Ъ \end{otherlanguage*}}

\def\Ts{\Phi}
\def\eTs{\Phi}

\def\kh{\mathcal{A}^{T}_{\tQ,\tW}}
\def\khrat{\mathcal{A}^{T}_{\tQ,\tW,\text{rat}}}
\def\khrate{\mathcal{A}^{T}_{\tQ,\tW,\emph{rat}}}
\def\kha{\mathcal{A}^{T,\omega\text{-nilp}}_{\tQ,\tW}}
\def\kharat{\mathcal{A}^{T,\omega\text{-nilp}}_{\tQ,\tW,\text{rat}}}
\def\kharate{\mathcal{A}^{T,\omega\emph{-nilp}}_{\tQ,\tW,\emph{rat}}}

\def\kzero{\mathcal{A}^{T}_{\tQ}}

\def\khan{\mathcal{A}^{T,\omega\text{-nilp}}_{\tQ,\tW,\bn}}

\def\khaloc{\mathcal{A}^{T,\omega\text{-nilp}}_{\tQ,\tW,\text{loc}}}

\def\ekha{\mathcal{A}^{T,\omega\emph{-nilp}}_{\tQ,\tW}}

\def\bps{\mathfrak{g}_{\tilde{Q},\tilde{W}}}

\def\tbpsw{\mathfrak{g}^{T,\omega\text{-nilp}}_{\tilde{Q},\tilde{W}}}
\def\tbpswn{\mathfrak{g}^{T,\omega\text{-nilp}}_{\tilde{Q},\tilde{W},\bn}}
\def\tbps{\mathfrak{g}^T_{\tilde{Q},\tilde{W}}}
\def\tbpsn{\mathfrak{g}^T_{\tilde{Q},\tilde{W},\bn}}

\def\point{(\cdot)}

\begin{document}

\title[The loop-nilpotent cohomological Hall algebra]{\Large{\textbf{The loop-nilpotent cohomological Hall algebra}}} 

\author[Shivang Jindal and Andrei Negu\cb t]{Shivang Jindal and Andrei Negu\cb t}

\address{École Polytechnique Fédérale de Lausanne (EPFL), Lausanne, Switzerland} 
\email{shivang.jindal@epfl.ch}

\address{École Polytechnique Fédérale de Lausanne (EPFL), Lausanne, Switzerland \newline \text{ } \ \ Simion Stoilow Institute of Mathematics (IMAR), Bucharest, Romania} 
\email{andrei.negut@gmail.com}
	
\begin{abstract} We give an explicit shuffle algebra model for the loop-nilpotent cohomological Hall algebra (CoHA) of a tripled quiver with canonical cubic potential. As consequences, we (1) relate the loop-nilpotent CoHA to the quantized Coulomb branch algebra of the corresponding quiver gauge theory, (2) show that the loop-nilpotent CoHA is supercommutative after specialization at $\hbar=0$, (3) give generators for both the loop-nilpotent CoHA and the full preprojective CoHA, and (4) obtain an explicit characterization of the BPS Lie algebra of the full preprojective CoHA via certain degree and divisibility conditions. This gives a new formula for the Kac polynomials of the quiver in terms of the dimensions of certain vector spaces of polynomials. We also prove a conjecture on the spherical generation of the localized shuffle algebra and show that for ADE quivers, the loop-nilpotent CoHA is the positive half of Drinfeld-Gavarini dual of the Yangian. \end{abstract}

\maketitle

\bigskip

\section{Introduction}
\label{sec:intro}

\medskip

\subsection{Motivation}
\label{sub:motivation intro}

Fix a quiver $Q$ with vertex set $I$ and arrow set $\edge$. We define the doubled quiver $\bar{Q}$ by adding an arrow $\dalpha : j \rightarrow i$ for every arrow $\alpha : i \rightarrow j$ in $\edge$, and we define the tripled quiver $\tQ$ by further adding a loop $\omega_i : i \rightarrow i$ at every vertex. With this in mind, for any $\bn = (n_i)_{i \in I} \in \nn$ we consider the affine space of $\bn$-dimensional representations of the tripled quiver
\begin{equation}
\label{eqn:representation intro}
\text{Rep}_{\bn}(\tQ) = \Big\{\underset{\omega_i}{\underset{\circlearrowleft}{\BC}}^{n_i} \xleftrightharpoons[\alpha]{\dalpha} \underset{\omega_j}{\underset{\circlearrowleft}{\BC}}^{n_j}\Big\}
\end{equation}
An important role in the present paper will be played by the closed subset
\begin{equation}
\label{eqn:closed subset intro}
\text{Rep}^{\omega\text{-nilp}}_{\bn}(\tQ) \hookrightarrow \text{Rep}_{\bn}(\tQ) 
\end{equation}
corresponding to the condition that all the loops $\{\omega_i\}_{i \in I}$ act by nilpotent operators. The algebraic group
$$
\GL_{\bn} = \prod_{i \in I} \GL_{n_i}(\BC)
$$ 
acts by conjugation on the affine space \eqref{eqn:representation intro}, and preserves the closed subset \eqref{eqn:closed subset intro}. This allows us to define the stack of loop-nilpotent $\bn$-dimensional quiver representations
\begin{equation}
\label{eqn:stacks intro}
\mathfrak{M}^{\omega\text{-nilp}}_{\bn}(\tQ) = \mathrm{Rep}^{\omega\text{-nilp}}_{\bn}(\tQ)/\GL_{\bn}
\end{equation}
We consider an algebraic torus $T \cong (\BC^*)^r$, and assume that we have chosen
\begin{equation}
\label{eqn:parameters intro}
\{\hbar, u_\alpha\}_{\alpha \in \edge} \in \ring = \BZ[\text{Lie}(T)]
\end{equation}
This choice induces an action of $T$ on the stack \eqref{eqn:stacks intro}, where the torus acts with the equivariant parameter $-u_\alpha$ on the arrow $\alpha : i \rightarrow j$, with the equivariant parameter $-\hbar + u_\alpha$ on the arrow $\dalpha : j \rightarrow i$ that is opposite to $\alpha$, and with the equivariant parameter $\hbar$ on all the loops $\omega_i$. The canonical cubic potential 
$$
\tW = \sum_{\alpha : i \rightarrow j} \Big( \dalpha \alpha \omega_i - \omega_j \alpha \dalpha \Big)
$$
therefore determines a $T$-invariant function on the stacks \eqref{eqn:stacks intro}. Following the general theory of \cite{KS}, we may use the data above to define the loop-nilpotent cohomological Hall algebra (CoHA for short)
\begin{equation}
\label{eqn:coha intro}
\kha = \bigoplus_{\bn \in \mathbb{N}^{I}}  H(\mathfrak{M}^{\omega\text{-nilp}}_{\bn}(\tilde{Q})/T, i^{!}_{\omega}\varphi_{\Tr(\tW)}\mathbb{Z}^{\vir}_{} )
\end{equation}
(see Subsection \ref{sub:coha} for details).   The purpose of the present paper is to study the cohomological Hall algebra \eqref{eqn:coha intro}, and to prove the following facts about it:

\medskip 

\begin{itemize}[leftmargin=*]

\item The shuffle realization of $\kha$, i.e. there is an isomorphism
\begin{equation}
\label{eqn:iso intro}
\kha \xrightarrow{\sim} \CS^+
\end{equation}
where $\CS^+$ is an explicit $\ring$-algebra of polynomials, see Subsection \ref{sub:small shuffle}. 

\medskip 

\item The supercommutativity of the specialization
$$
\kha \Big|_{\hbar = 0}
$$
in Corollary \ref{nilCoHAcommute}.

\medskip 

\item A set of generators $\{e_{\bn,g}\}$ of $\CS^+ \cong \kha$, for $\bn \in \nn$ and $g \in \ring[z_{i1},\dots,z_{in_i}]_{i\in I}^{\sym}$ \footnote{We remark that generators of nilpotent CoHA's have been studied previously in \cite{schiffmanncohagenerators}, but our nilpotency condition is different. In terms of tripled quivers, our nilpotency is about the added loops being nilpotent, while \loccit studies the setting where the nilpotency condition is on the original arrows of the quiver. Our result coincide with that of \loccit in the case of Jordan quiver (see Section \ref{sec:examples}).}.

\medskip 

\item The proof of a conjecture of \cite{jindalnegutbps} and \cite{NGen} on the spherical generation of 
$$
\CS^+_{\text{loc}} \cong \khaloc
$$
(the subscript ``loc'' means the respective $\ring$-algebras are tensored with $\text{Frac}(\ring)$).

\medskip 

\item The construction of a shifted double loop-nilpotent CoHA, as follows. Make
\begin{equation}
\label{eqn:double intro}
\CS^{\br}_{\text{loc}} = \CS^+_{\text{loc}} \otimes (\text{Cartan subalgebra}) \otimes \CS^{+,\text{op}}_{\text{loc}}
\end{equation}
into an algebra using a twist by $\br \in \zz$, see Subsection \ref{sub:double} for details. Then let
\begin{equation}
\label{eqn:shuffle intro}
\CS^{\br} = \Big(\text{shifted double of }\CS^+ \Big)
\end{equation}
be the $\ring$-subalgebra of \eqref{eqn:double intro} generated by the $e_{\bn,g} \in \CS^+$ and by their counterparts in the opposite algebra $\CS^{+,\text{op}}$. Because of the isomorphism \eqref{eqn:iso intro}, we may interpret the $\ring$-algebra \eqref{eqn:shuffle intro} as our shifted double loop-nilpotent CoHA.

\medskip 

\item The construction of a surjective homomorphism 
$$
\Big(\text{shifted double of }\kha \Big) \twoheadrightarrow \CH_{\bd|\bk,\bell}
$$
where the right-hand side denotes the quantized Coulomb branch algebra associated to the (framed, equivariant) quiver gauge theory in \cite{BFN}, see Subsection \ref{sub:coulomb}.

\medskip 

\item A precise relationship between the BPS Lie algebra
$$
\tbpsw \subset \kharat
$$
(see Subsection \ref{sub:cohomological integrality}) and of its analog without the loop-nilpotency condition, which implies that the loop-nilpotent and usual Kac polynomials of $Q$ are equal up to multiplication by $\hbar$ (the subscript ``rat'' refers to tensoring with $\BQ$).

\medskip 

\item Combining with \cite{jindalnegutbps}, we obtain a full description of $\tbpsw$ and $\tbps$ (and thus of the Kac polynomials of $Q$) in terms of symmetric polynomials (see Subsection \ref{sub:bps explicit}) \footnote{It is known due to \cite{botta2023okounkovs} and \cite{schiffmann2024cohomologicalhallalgebrasquivers} that the BPS Lie algebra $\tbps$ is isomorphic to the positive half of the Maulik–Okounkov Lie algebra (\cite{moyangian}), that is defined using stable envelopes. It can also be identified with the positive half of a generalized Kac–Moody Lie algebra, whose Chevalley generators are given by the intersection cohomology of the highly singular good moduli spaces of representations of the preprojective algebra (see \cite{DHS}). Although these descriptions are conceptually interesting, they are not readily amenable to explicit computation. Our results thus provide an effective method for computing this Lie algebra explicitly.}.

\end{itemize}

\medskip 

\noindent Finally, in Section \ref{sec:examples}, we provide examples of the constructions above for $Q$ a simply-laced Dynkin quiver (in which case we identify the loop-nilpotent CoHA with the Drinfeld-Gavarini dual of the Yangian) and for $Q$ a Jordan quiver (in which case we encounter the Borel-Moore homology of the semi-nilpotent commuting stack). We will now provide additional details on the results above.

\medskip
 
\subsection{Shuffle algebras} 
\label{sub:shuffle intro}

We recall the following $\ring$-algebra homomorphism
\begin{equation}
\label{eqn:intro 1}
\kha \rightarrow \kzero
\end{equation}
where the right-hand side denotes the CoHA with zero potential and without the loop-nilpotency condition. Using well-known arguments, we show that the map \eqref{eqn:intro 1} is injective in Subsection \ref{sub:injection}. As observed in \cite{KS}, we have an isomorphism
\begin{equation}
\label{eqn:intro 2}
\kzero \xrightarrow{\sim} \CV^+ = \bigoplus_{\bn = (n_i)_{i \in I} \in \nn} \ring[z_{i1},\dots,z_{in_i}]^{\sym}_{i \in I}
\end{equation}
where the vector space in the right-hand side is made into an $\ring$-algebra using the Feigin-Odesskii shuffle product \eqref{eqn:mult}. Following the trigonometric case studied in \cite{JNCoulomb}, we construct an explicit $\ring$-subalgebra
\begin{equation}
\label{eqn:intro 3}
\CS^+ \subset \CV^+
\end{equation}
consisting of polynomials that satisfy the divisibility conditions in Definition \ref{def:shuffle int}. Combining \eqref{eqn:intro 1}, \eqref{eqn:intro 2}, \eqref{eqn:intro 3} allows us to state our first main result.

\medskip 

\begin{theorem}
\label{thm:intro coha}

The composition of the maps \eqref{eqn:intro 1} and \eqref{eqn:intro 2} induces an isomorphism
\begin{equation}
\label{eqn:intro coha}
\ekha \xrightarrow{\sim} \CS^+
\end{equation}
under the Assumption \soft of \eqref{eqn:assumption soft}.

\end{theorem}

\medskip 

\noindent A key tool in our proof of the isomorphism \eqref{eqn:intro coha} is the fact that the shuffle algebra $\CS^+$ is generated as an $\ring$-algebra by the explicit elements $\{e_{\bn,g}\}_{\bn \in \nn, g \in \ring[z_{i1},\dots,z_{in_i}]_{i \in I}^{\sym}}$ of \eqref{eqn:special}. This will allow us to prove \cite[Conjecture 4.3]{jindalnegutbps} and \cite[Conjecture 2.12]{NGen}, namely the following result.

\medskip 

\begin{theorem}
\label{thm:intro spherical}

Let $\field = \emph{Frac}(\ring)$ and consider the localized shuffle algebra
$$
\CS^+_{\emph{loc}} = \CS^+ \otimes_{\ring} \field 
$$
Under Assumptions \hard and \soft of \eqref{eqn:assumption hard} and \eqref{eqn:assumption soft}, $\CS^+_{\emph{loc}}$ is generated by
$$
\Big\{e_{i,k} = z_{i1}^k \Big\}_{i \in I, k \geq 0}
$$

\end{theorem}

\medskip

\subsection{Doubles and Coulomb branches}
\label{sub:doubles intro}

In the seminal paper \cite{BFN}, Braverman-Finkelberg-Nakajima associated to any dimension vectors $\bd \in \nn$ (for the gauge vertices) and $\bk,\bell \in \nn$ (for the framing vertices) an $\ring$-algebra
$$
\CH_{\bd|\bk,\bell}  
$$
called the quantized Coulomb branch algebra. They also showed how to realize this algebra inside a ring of difference operators (see also \cite{GKLO, KWWY} for closely related constructions), which has been related to shuffle algebras by the construction of Finkelberg-Frassek-Tsymbaliuk (\cite{FT2, FrT,Tsymbaliuk}). With this in mind, we prove the following result analogous to the main theorem of \cite{JNCoulomb}.

\medskip 

\begin{theorem}
\label{thm:intro main}

For any $\bd,\bk,\bell \in \nn$, there is a surjective $\ring$-algebra homomorphism
\begin{equation}
\label{eqn:main}
\Big(\bd^\vee-\bk-\bell\text{ shifted double of }\ekha \Big) \twoheadrightarrow \CH_{\bd|\bk,\bell} 
\end{equation}
under the Assumption \soft of \eqref{eqn:assumption soft}, where $\bd^\vee$ is defined in \eqref{eqn:vee}.
    
\end{theorem}

\medskip 

\noindent One can think of \eqref{eqn:main} as an explicit mathematical incarnation of the relationship between Higgs branches and Coulomb branches of quiver gauge theories (\cite{CL, RSYZ}), as the CoHA is the natural object which acts on the cohomology groups of Nakajima quiver varieties (algebro-geometric relatives of Higgs branches).

\medskip 

\noindent Note that we do not provide a direct geometric description of the so-called shifted double which appears in the domain of the map \eqref{eqn:main}. Instead, in Subsection \ref{sub:double} we use the isomorphism \eqref{eqn:intro coha} to define the shifted double as a particular subalgebra
$$
\CS^{\br} \subset \CS^{\br}_{\text{loc}}
$$
for any $\br \in \zz$. The right-most algebra admits a triangular decomposition
$$
\CS^{\br}_{\text{loc}} = \CS^+_{\text{loc}} \otimes (\text{Cartan subalgebra}) \otimes \CS^-_{\text{loc}}
$$
where the commutation relations between the algebras $\CS^+_{\text{loc}}$ and $\CS^-_{\text{loc}} = \CS^{+,\text{op}}_{\text{loc}}$ explicitly depend on the shift $\br$ (as defined originally in \cite{BK, KWWY}), see \eqref{eqn:rel double 2}-\eqref{eqn:def xi}.

\medskip 

\subsection{BPS Lie algebras and Kac polynomials} 
\label{sub:bps intro}

Consider the
$$
\ring_{\BQ} = \BQ[\text{Lie}(T)]\text{-algebra} \quad \kharat = \kha \otimes_{\BZ} \BQ 
$$
The well-known constructions of \cite{davison2020cohomological, davison2022integrality} apply to the loop-nilpotent setting, and give rise to the so-called BPS Lie (super)algebra
\begin{equation}
\label{eqn:bps intro}
\tbpsw \subset \kharat
\end{equation}
The importance of the BPS Lie algebra to the theory is to provide a PBW basis of the loop-nilpotent CoHA, in the sense of the existence of a $\ring_{\BQ}$-module isomorphism
$$
\kharat \cong \text{Sym} \left(\tbpsw[u] \right) 
$$
(see Subsection \ref{sub:cohomological integrality}). In \cite{jindalnegutbps}, we identified the image of the Lie subalgebra \eqref{eqn:bps intro} under the injective $\ring$-algebra homomorphism
$$
\kharat \hookrightarrow \CV^+_{\BQ} = \bigoplus_{\bn = (n_i)_{i \in I} \in \nn} \ring_{\BQ}[z_{i1},\dots,z_{in_i}]^{\sym}_{i \in I}
$$
with the set of $E(z_{i1},\dots,z_{in_i})_{i \in I} \in \CS^+_{\BQ} = \CS^+ \otimes_{\BZ} \BQ$ such that for any partition
\begin{equation}
\label{eqn:intro partition}
\bn = \bn^1 + \dots + \bn^k
\end{equation}
(for arbitrary $k \geq 1$ and $\bn^1,\dots,\bn^k \in \nn \backslash \b0$, where $\bn = (n_i)_{i \in I}$), the polynomial
\begin{equation}
\label{eqn:intro add y}
E(y_1+z_{i,1},\dots,y_1+z_{i,n_i^1}, \dots,y_k+z_{i,n_i-n_i^k+1}, \dots, y_k+z_{i,n_i})_{i \in I}
\end{equation}
satisfies the following degree bound (in $y_1,\dots,y_k$)
\begin{equation}
\label{eqn:intro degree bound}
\text{deg}_{y_1,\dots,y_k}(E) \leq \frac {1-k}2 - \sum_{1\leq a < b \leq k} (\bn^a,\bn^b)'
\end{equation}
(the right-hand side features the bilinear form of \eqref{eqn:modified euler form}). Moreover, we compare the objects above with their ``full'' counterparts, i.e. the cohomological Hall algebras defined without the loop-nilpotency condition. There is a commutative diagram
 \begin{equation}
\label{eqn:commutative intro}
\begin{tikzcd}
\tbpsw	 & \kharat \\
\tbps & \khrat
    \arrow[from=1-1, to=1-2]
	\arrow[from=1-1, to=2-1]
	\arrow[from=1-2, to=2-2]
	\arrow[from=2-1, to=2-2]
\end{tikzcd}
\end{equation}
with all maps injective, as long as the geometric genericity assumption \eqref{eqn:assumption geometric} holds. If it holds, then we prove in Subsection \ref{sub:support} that the left vertical map identifies
\begin{equation}
\label{eqn:hbar intro}
\tbpsw = \hbar \tbps
\end{equation}
which has the following important consequence concerning the graded dimensions of BPS Lie (super)algebras. It is well-known from \cite{davison2022integrality} and \cite{Mozgovoy} that
$$
A_Q(t^{-1}) = (1-t)^{r} \text{grdim}_{\BQ}(\tbps) 
$$
where the left-hand side is the (generating function of) Kac polynomial of the quiver $Q$, and the right-hand side features the $\nn \times \BZ$ graded dimension of the BPS Lie algebras. In other words, for all $\bn \in \nn$ the equation above stands for the equality
\[ 
A_{Q,\bn}(t^{-1}) = (1-t)^{r} \sum_{d \in \mathbb{Z}} \dim_{\BQ} \Big(\text{degree }d \text{ part of }\fg^T_{\tilde{Q},\tilde{W},\bn} \Big) t^{\frac d2}
\] 
The grading in the right-hand side will be defined in Section \ref{sec:coha} as an appropriate shift of the cohomological grading, and the quantity $(1-t)^{-r}$ is simply the graded dimension of $\ring_{\BQ} = \BQ[\text{Lie}(T)]$. In fact, BPS Lie algebras are well-known to be free over $\ring_{\BQ}$. If we define the analogous notions for the loop-nilpotent version
$$
A^{\omega\text{-nilp}}_Q(t^{-1}) := (1-t)^{r} \text{grdim}_{\BQ}(\tbpsw) 
$$
then \eqref{eqn:hbar intro} implies the following equation:
$$
A^{\omega\text{-nilp}}_Q(t^{-1}) = t A_Q(t^{-1})
$$
Combining this equality with the main result of \cite{jindalnegutbps} yields the following formula for the Kac polynomial of a quiver $Q$, which to the best of our knowledge is new.

\medskip 

\begin{corollary}
\label{cor:kac}

For any $\bn = (n_i)_{i \in I} \in \nn$, we have
\begin{multline}
\label{eqn:kac}
t^{(\bn,\bn)'+1} A_{Q,\bn}(t^{-1}) = \\ (1-t)^r \sum_{d \in \BZ} \dim_{\BQ} \left\{\begin{split} &\text{degree }d\text{ polynomials }E(z_{i1},\dots,z_{in_i})_{i \in I} \text{ as in} \\ &\text{Definition \ref{def:shuffle int}, that satisfy the bounds \eqref{eqn:intro add y}-\eqref{eqn:intro degree bound}} \end{split} \right\} t^{\frac d2}
\end{multline}
with the degree of polynomials determined by $\deg z_{ia} = \deg u_\alpha = \deg \hbar = 2$. \footnote{The overall shift by a power of $t$ in the LHS is due to the convention of an overall cohomological degree shift by $(\bn,\bn)^{\prime}$ in $\khan$.} 

\end{corollary}

\medskip 

\noindent Note that \eqref{eqn:kac} gives a formula for the \emph{coefficients} of Kac polynomials as dimensions of certain vector spaces of polynomials, as opposed from the original definition of the \emph{values} of Kac polynomials as counting absolutely indecomposable quiver representations over finite fields.

\medskip 

\subsection{Relation to other work}

There are several papers on similar topics that have recently been or are currently being written independently of the present one. In \cite{MW}, the authors define an injective analogue of the map \eqref{eqn:main}, with the codomain replaced by the (suitably defined) limit as $\bd \rightarrow \infty$ of the Coulomb branch algebras. Using this connection, they recover the Kac polynomial of the quiver via Hua's formula from the monopole formula for quantized Coulomb branches and prove the conjecture of \cite{NGen} using spherical generation for Coulomb branches. Combining our results with those of \cite{MW} provides an isomorphism between the loop-nilpotent CoHA and the strictly positive part of the $\bd \rightarrow \infty$ limit of Coulomb branch algebras (which \loccit interpret rigorously using zastava spaces).

\medskip 

In \cite{BT}, the authors define an action of the loop-nilpotent CoHA on the cohomology of quasimap moduli spaces, the latter being the natural modules for Coulomb branch algebras; this yields a geometric connection between the CoHA and Coulomb branches. They also extend the CoHA action to an action of shifted Yangian, in the sense of \cite{cao2026shiftedquantumgroupscritical}. In \cite[Section 2.3]{DHKSV}, the authors study an analogue of the loop-nilpotent locus in the very general setting of 2-Calabi-Yau categories, which in particular also applies to the quivers studied in the present paper. Finally, \cite{H} studies the nilpotent versions of the CoHA of preprojective algebra and proves Corollary \ref{nilCoHAcommute} in the particular case of the Jordan quiver, through different means from ours.

\medskip

\subsection{Acknowledgements} We would like to thank Tommaso Maria Botta, Ben Davison, Dinakar Muthiah, Tudor Pǎdurariu, Olivier Schiffmann, Alexander Tsymbaliuk and Alex Weekes for many interesting conversations. We gratefully acknowledge the support of the Swiss National Science Foundation grant 10005316.

\bigskip 

\section{Shuffle algebras}
\label{sec:shuffle}

\medskip 

\noindent We review the main results of \cite{JNCoulomb}, and translate them from the $K$-theoretic to the cohomological Hall algebra language. We prove the supercommutativity of the shuffle algebra at $\hbar = 0$. The parameters denoted by $q$ and $t_\alpha$ in \loccit are the exponentials of the parameters denoted by $\hbar$ and $u_\alpha$ in the present paper. 

\medskip

\subsection{Conventions}
\label{sub:conventions}

In the present paper, the set $\BN$ will contain 0. We recall that $Q$ is a quiver with vertex set $I$. We write $\b0 = (0,\dots,0) \in \nn$ and let $\bs^i \in \nn$ be the $I$-tuple with a single 1 at position $i$ and 0 everywhere else. The adjacency matrix of the quiver $Q$ induces the bilinear form
\begin{equation}
\label{eqn:euler form}
\zz \times \zz \xrightarrow{( \cdot, \cdot )} \BZ
\end{equation}
determined by
$$
( \bs^i, \bs^j ) = 2\delta_{ij} - \Big|\text{arrows }i \rightarrow j\Big| - \Big|\text{arrows }j \rightarrow i\Big|
$$
Given $\bd \in \zz$, we will write $\bd^\vee$ for the $I$-tuple of integers whose entries are
\begin{equation}
\label{eqn:vee}
d^\vee_i = 2d_i - \sum_{j \in I} d_j \left( \Big|\text{arrows }i \rightarrow j\Big| + \Big|\text{arrows }j \rightarrow i\Big| \right)
\end{equation}
so that for any $\bn \in \zz$ the usual dot product of $I$-tuples satisfies the equation
$$
\bd^\vee \cdot \bn = (\bd,\bn)
$$
For any $\bm = (m_i)_{i \in I}, \bn = (n_i)_{i \in I}$, we write $\bm \leq \bn$ if $m_i \leq n_i$ for all $i \in I$. Let
$$
|\bn| = \sum_{i \in I} n_i
$$
We will also encounter the modified version of the bilinear form \eqref{eqn:euler form}, namely
\begin{equation}
\label{eqn:modified euler form}
\zz \times \zz \xrightarrow{( \cdot, \cdot )'} \BZ
\end{equation}
determined by
$$
( \bs^i, \bs^j )' = - \Big|\text{arrows }i \rightarrow j\Big| - \Big|\text{arrows }j \rightarrow i\Big|
$$
which is actually the form defined in \cite{jindalnegutbps} in the case of tripled quivers.

\medskip 

\subsection{Parameters}
\label{sub:parameters}

Consider the quiver $Q$ of Subsection \ref{sub:conventions}, and let $\edge$ denote its set of arrows. We will work over a field $\field$ of characteristic 0 endowed with elements
\begin{equation}
\label{eqn:parameters}
\left\{  \hbar, u_{\alpha} \right\}_{\alpha \in \edge} \subset \field
\end{equation}
and define the subring 
$$
\ring \subset \field \quad \text{generated by }\left\{ \hbar, u_{\alpha} \right\}_{\alpha \in \edge}
$$ 
Consider the doubled quiver $\dQ$ with vertex set $I$, whose arrow set $\dedge$ consists of
\begin{equation}
\label{eqn:double parameters}
i \xleftrightharpoons[u_{\alpha}]{\hbar-u_{\alpha}} j
\end{equation}
for every arrow $i \rightarrow j$ in $\edge$. The symbols $u_\alpha, \hbar - u_\alpha$ above represent parameters in $\ring$ associated to the arrows in the double quiver. Alternatively, if we write $\dalpha : j \rightarrow i$ for the opposite arrow to any $\alpha : i \rightarrow j$ in $\dedge$, then the parameter of $\dalpha$ is defined as
\begin{equation}
\label{eqn:opposite parameters}
u_{\dalpha} = \hbar - u_{\alpha}
\end{equation}
In the present paper, we will adapt from \cite{JNCoulomb} the following genericity assumptions on the parameters \eqref{eqn:parameters}: 
\begin{equation}
\label{eqn:assumption hard}
\text{Assumption \hard}: \quad \hbar \neq 0, \text{ and}
\end{equation}
$$
\text{the equation } \sum_{\alpha \in \edge} ( x_{\alpha} u_{\alpha} + y_{\alpha} u_{\dalpha} ) = 0 \text{ has no non-trivial solution }x_\alpha, y_\alpha \in \BZ_{\geq 0}
$$
\begin{equation}
\label{eqn:assumption soft}
\text{Assumption \soft}: \quad \hbar \neq 0, \text{ and}
\end{equation}

\begin{itemize}[leftmargin=*]

\item $\left\{u_{\alpha} + \BZ_{< 0}\hbar, -u_{\beta} + \BZ_{> 0} \hbar\right\}_{i \xleftrightharpoons[\alpha]{\beta} j} \cap \left\{u_{\alpha} + \BZ_{\geq 0}\hbar, -u_{\beta} + \BZ_{\leq 0}\hbar \right\}_{i \xleftrightharpoons[\alpha]{\beta} j} = \varnothing$, $\forall i,j \in I$;

\medskip 

\item $x\left(u_\alpha +\BZ_{<0} \hbar \right) = y \left(u_\beta +\BZ_{\geq 0}\hbar \right) \text{ for }x,y 
\in \BZ \text{ implies }x = y = 0, \forall \alpha,\beta \in \edge_{\text{loop}}$ (above, $\edge_{\text{loop}} \subset \edge$ denotes the set of loops of $Q$)

\end{itemize}

\begin{equation}
\label{eqn:assumption geometric}
\text{Geometric assumption}: \quad \hbar \neq 0 \text{ and } u_{\alpha} \notin \BZ \hbar, \forall \alpha \in \Omega_{\text{loop}}
\end{equation}

\medskip 

\noindent In fact, the original setting for the conditions above is when $\hbar, u_\alpha$ denote the canonical images of a ring homomorphism
\begin{equation}
\label{eqn:geometric parameters}
\frac {\BZ[\hbar,u_\alpha]_{\alpha \in \edge}}{\Big(\text{various $\BZ$-linear combinations of }\hbar,u_\alpha\Big)} \rightarrow \BF 
\end{equation}
Indeed, this is the case for the geometric applications considered in the present paper, in which case $\BF$ will be the fraction field of the representation ring of an algebraic torus $T$, and $\hbar,u_\alpha$ are various equivariant parameters (i.e. linear functions on the Lie algebra of $T$). We will occasionally blur the lines between the general definition of an arbitrary field with parameters \eqref{eqn:parameters} and the particular case of \eqref{eqn:geometric parameters}, but we recommend to keep in mind the fact that all known applications of the results in the present paper pertain to the setting \eqref{eqn:geometric parameters}.

\medskip 

\subsection{Big shuffle algebras}
\label{sub:big shuffle}

Corresponding to a quiver $Q$ with parameters as in Subsection \ref{sub:parameters}, we associate the following rational function (the minus signs are used to keep with the conventions in \cite{jindalnegutbps})
\begin{equation}
\label{eqn:zeta}
\zeta_{ij}(x) = \left( \frac {x - \hbar}x \right)^{\delta_{ij}} \prod_{(\alpha : i \rightarrow j) \in \bar{\edge}} \left( - x - u_\alpha \right) \in \ring(x)
\end{equation}
This choice allows us to place associative algebra structures on
\begin{align*}
&\CV^+ = \bigoplus_{\bn = (n_i \geq 0)_{i \in I} \in \nn} \CV_{\bn}, \qquad \quad \ \CV_{\bn}= R[z_{i1},\dots,z_{in_i}]_{i \in I}^{\sym} \\
&\CV^+_{\text{loc}} = \bigoplus_{\bn = (n_i \geq 0)_{i \in I} \in \nn} \CV_{\text{loc},\bn}, \qquad \CV_{\text{loc},\bn}=\BF[z_{i1},\dots,z_{in_i}]_{i \in I}^{\sym}
\end{align*}
(above, ``sym'' refers to polynomials which are color-symmetric, i.e. symmetric in $z_{i1},\dots,z_{in_i}$, $\forall i \in I$). Explicitly, the multiplication on the sets above is given by
\begin{equation}
\label{eqn:mult}
E( z_{i1}, \dots, z_{i n_i})_{i \in I} * E'(z_{i1}, \dots,z_{i n'_i})_{i \in I} = 
\end{equation}
$$
\textrm{Sym} \left[ \frac {E(z_{i1}, \dots, z_{in_i}) E'(z_{i,n_i+1}, \dots, z_{i,n_i+n'_i})}{\bn! \bn'!}
\prod_{i,j \in I} \mathop{\prod_{1 \leq a \leq n_i}}_{n_j < b \leq n_j+n_j'} \zeta_{ij} \left( z_{ia} - z_{jb} \right) \right]
$$
with ``Sym'' in \eqref{eqn:mult} referring to summing over the
\begin{equation*}
(\bn+\bn')! := \prod_{i\in I} (n_i+n'_i)!
\end{equation*}
permutations of the variables $\{z_{i1}, \dots, z_{i,n_i+n'_i}\}$ for each $i$ independently. We call $\CV^+$ the \emph{big shuffle algebra} and $\CV^+_{\text{loc}}$ its localized version. Formula \eqref{eqn:mult} is a particular case of a Feigin-Odesskii shuffle product.

\medskip

\subsection{Small shuffle algebras}
\label{sub:small shuffle}

We now introduce certain subalgebras of $\CV^+$, $\CV^+_{\text{loc}}$.

\medskip 

\begin{definition} 
\label{def:shuffle loc}

Let $\CS^+_{\emph{loc}} \subset \CV^+_{\emph{loc}}$ denote the set of $E(z_{i1},\dots,z_{in_i})_{i \in I}$ such that
\begin{equation}
\label{eqn:divisibility loc}
E\Big|_{z_{i2} = z_{i1} + \hbar} \quad \text{ is divisible by } \prod_{(\alpha : i \rightarrow j) \in \bar{\edge}} (z_{jb} - z_{i1} - u_\alpha)
\end{equation}
for all $i,j \in I$ and any $b \geq 1$ (we require $b>2$ if $i=j$).

\end{definition} 

\medskip 

\noindent The set $\CS^+_{\text{loc}}$ is $\nn$ graded, with the components denoted by $\CS_{\text{loc},\bn}$. The following result was conjectured in \cite{jindalnegutbps}, and we will prove it in the present paper.

\medskip 

\begin{theorem}
\label{thm:shuffle loc}

Under Assumption \hard of \eqref{eqn:assumption hard}, the set $\CS^+_{\emph{loc}}$ coincides with the
\begin{equation}
\label{eqn:spherical}
\field\text{-algebra generated by }\left\{ e_{i,k} = z_{i1}^k \right\}_{i \in I, k \in \BN} \subset \CV^+_{\emph{loc}}
\end{equation}
where $z_{i1}^k$ is interpreted as a polynomial in a single variable of color $i$ (even if $k=0$).

\end{theorem}

\medskip 

\noindent The following construction is the cohomological version of that of \cite{JNCoulomb}, itself generalizing the situation of the Jordan quiver from \cite{Integral}). We will refer to
\begin{equation}
\label{eqn:composition}
P = \left\{ n_i = n_i^{(1)}+\dots+n_i^{(d_i)} \right\}_{i \in I}
\end{equation}
as an $I$-composition of $\bn = (n_i)_{i \in I} \in \nn$. If moreover
$$
n_i^{(1)} \geq \dots \geq n_i^{(d_i)}
$$
for all $i \in I$, then we call $P$ an $I$-partition.

\medskip 

\begin{definition} 
\label{def:shuffle int}

Let $\CS^+ \subset \CV^+$ denote the set of  polynomials $E(z_{i1},\dots,z_{in_i})_{i \in I}$ such that for any $I$-composition $P$ as in \eqref{eqn:composition}, the specialization
\begin{equation}
\label{eqn:specialization}
\emph{Spec}_P(E) = E \left(x_{ia},x_{ia}+ \hbar,\dots,x_{ia} + (n_i^{(a)}-1)\hbar\right)_{i \in I, a \in \{1,\dots,d_i\}}
\end{equation}
is divisible in the ring $\ring[x_{i1},\dots,x_{id_i}]^{\emph{sym}}_{i \in I}$ by
\begin{multline}
\label{eqn:factor}
\prod_{i \in I} \left( \hbar^{n_i} \prod_{a=1}^{d_i} n_i^{(a)}! \right) \prod_{(\alpha : i \rightarrow j) \in \edge} \mathop{\prod_{1 \leq a \leq d_i}}_{1 \leq b \leq d_j} \prod_{c \in \BZ} \left(x_{jb}-x_{ia}+c\hbar - u_\alpha\right)^{\chi_{n_i^{(a)},n_j^{(b)}}(c)}
\end{multline}
where for any $k,k' \geq 1$, we let $\chi_{k,k'}(c) = \mathrm{max}(0,\mathrm{min}(k, k', c+k-1, k'-c))$. By the color-symmetry of $E$, it is enough to restrict attention to $I$-partitions $P$ above.

\end{definition}

\medskip 

\noindent We call $\CS^+$ thus defined the \emph{integral shuffle algebra}. It is also graded by $\nn$, with the components denoted by $\CS_{\bn}$. The following result is completely analogous to \cite[Theorem 2.6]{JNCoulomb} (also \cite[Proposition 3.9]{Integral}), so we only sketch its proof.

\medskip 

\begin{theorem}
\label{thm:shuffle int}

Under the geometric assumption \eqref{eqn:assumption geometric}, the set $\CS^+$ coincides with the $\ring$-subalgebra of $\CV^+$ generated by
\begin{equation}
\label{eqn:special}
e_{\bn,g} = g(z_{i1},\dots,z_{in_i})_{i \in I} \prod_{i \in I} \prod_{1 \leq a , b \leq n_i} \left(z_{ib}-z_{ia}+\hbar\right)
\end{equation}
as $\bn = (n_i)_{i \in I}$ ranges over $\nn$ and $g$ ranges over $\ring[z_{i1},\dots,z_{in_i}]^{\emph{sym}}_{i \in I}$.

\end{theorem}

\medskip 

\begin{proof} \emph{(sketch):} Denote the $\ring$-subalgebra generated by the $e_{\bn,g}$'s by
\begin{equation}
\label{eqn:subalgebra generated by}
\bar{\CS}^+ \subset \CV^+
\end{equation}
It is easy to see that $e_{\bn,g} \in \CS^+$ for all $\bn,g$, since we have $\text{Spec}_P(e_{\bn,g}) = 0$ whenever $P$ is an $I$-partition with at least one part of size $\geq 2$. Since $\CS^+$ is an $\ring$-subalgebra of $\CV^+$ (this is proved as in the proof of Proposition \ref{prop:spec is comm} below), we conclude that
\begin{equation}
\label{eqn:star}
\bar{\CS}^+ \subseteq \CS^+
\end{equation}
To establish the opposite inclusion, we consider for any $I$-partition $P$ the set
$$
\CS^+_P = \bigcap_{P' > P} \text{Ker}(\text{Spec}_{P'}) \subseteq \CS^+
$$
where $P' \geq P$ means that the $i$-th constituent partition of $P'$ is lexicographically greater than or equal to the $i$-th constituent partition of $P$ for all $i \in I$.  

\medskip

\begin{claim}
\label{claim}

For any $I$-partition $P$ and any $E \in \CS^+_P$, there exists $E' \in \bar{\CS}^+ \cap \CS^+_P$ such that
\begin{equation}
\label{eqn:claim}
\emph{Spec}_P(E) = \emph{Spec}_P(E')
\end{equation}

\end{claim}

\medskip 

\noindent By descending induction with respect to the partial order on $I$-partitions defined above, Claim \ref{claim} allows us to recursively associate to any $E \in \CS^+$ an element $E' \in \bar{\CS}^+$ such that \eqref{eqn:claim} holds for all $I$-partitions $P$. When $P$ is the finest $I$-partition with all parts equal to 1 (which is minimal with respect to the partial order), we have that $\text{Spec}_P = \text{Id}$, which implies that $E = E'$. This proves the opposite inclusion to \eqref{eqn:star} and thus completes the proof of Theorem \ref{thm:shuffle int}. As for Claim \ref{claim}, it is proved exactly like \cite[Claim 5.2]{JNCoulomb}, since the linear factors involved in the divisibility condition of Definition \ref{def:shuffle int} are the ``rational'' counterparts of the ``trigonometric'' factors involved in the divisibility condition of \cite[Definition 2.5]{JNCoulomb}. \end{proof}

\medskip 

\begin{proof} \emph{of Theorem \ref{thm:shuffle loc}:} The natural identification
\begin{equation}
\label{eqn:map 0}
\CV^+ \otimes_{\ring} \field = \CV^+_{\text{loc}}
\end{equation}
induces a map
\begin{equation}
\label{eqn:map}
\CS^+ \otimes_{\ring} \field \rightarrow \CS^+_{\text{loc}}
\end{equation}
because over $\field$, the $P = 2\bs^i + \bs^j + \dots$ case of the divisibility condition by \eqref{eqn:factor} is precisely equivalent to \eqref{eqn:divisibility loc}. The map \eqref{eqn:map} is injective, due to the fact that its domain and codomain are subsets of the respective sides of equation \eqref{eqn:map 0}. To prove the surjectivity of the map \eqref{eqn:map}, we must establish the following statement: consider any $E$ satisfying \eqref{eqn:divisibility loc}; prove that some $\ring$-multiple of $E$ satisfies the conditions in Definition \ref{def:shuffle int}. Because of the freedom in choosing said multiple, we may always ensure divisibility by arbitrary elements of $\ring$. As for divisibility by the other factors in \eqref{eqn:factor}, we note that the products
\begin{multline}
\label{eqn:one factor}
\prod_{(\alpha : i \rightarrow j) \in \edge} \prod_{c \in \BZ} \left(x_{jb}-x_{ia}+c\hbar - u_\alpha\right)^{\chi_{n_i^{(a)},n_j^{(b)}}(c)} \\ \prod_{(\alpha : j \rightarrow i) \in \edge} \prod_{c \in \BZ} \left(x_{jb}-x_{ia}-c\hbar + u_\alpha\right)^{\chi_{n_j^{(b)},n_i^{(a)}}(c)}
\end{multline}
are coprime for various $(i,a) \neq (j,b)$, and thus it suffices to establish the fact that $\text{Spec}_P(E)$ is divisible by \eqref{eqn:one factor} for any fixed $(i,a) \neq (j,b)$. Assume without loss of generality that $n_i^{(a)} \geq n_j^{(b)}$. Then conditions \eqref{eqn:divisibility loc} for the pairs of variables 
$$
(z_{i1},z_{i2}), \ (z_{i2},z_{i3}), \ \dots, \ (z_{i,n_i^{(a)}-1},z_{in_i^{(a)}})
$$
imply that $\text{Spec}_P(E)$ is divisible by
\begin{equation}
\label{eqn:two factor}
\prod_{(\alpha : i \rightarrow j) \in \bar{\edge}} \prod_{r=0}^{n_i^{(a)}-2} \prod_{s=0}^{n_j^{(b)}-1} (x_{jb} - x_{ia} + (s-r)\hbar - u_\alpha)
\end{equation}
The fact that the products \eqref{eqn:one factor} and \eqref{eqn:two factor} are equal follows from the identity
\begin{equation}
\label{eqn:explicit chi}
\chi_{k,k'}(c) = \begin{cases} \Big| \{ (r,s) \in \{0,\dots,k-1\} \times \{0,\dots, k'-1\} \text{ s.t. } s-r = c  \} \Big| &\text{if }c > 0 \\ \Big| \{ (r,s) \in \{0,\dots,k-1\} \times \{0,\dots, k'-1\} \text{ s.t. } s+1-r = c \} \Big| &\text{if }c \leq 0 \end{cases}
\end{equation}
which was proved in \cite[Subsection 5.1]{JNCoulomb}. Having proved \eqref{eqn:map}, it follows from Theorem \ref{thm:shuffle int} that $\CS^+_{\text{loc}}$ is generated over $\field$ by the shuffle elements $e_{\bn,g}$ (note that the geometric assumption \eqref{eqn:assumption geometric} is weaker than assumption \hard of \eqref{eqn:assumption hard}). 

\medskip 

\noindent Therefore, it remains to prove that the elements $e_{\bn,g}$ of \eqref{eqn:special} can be written as sums of products of $e_{i,k}$'s with coefficients in $\field$. To this end, we will mimic the argument of \cite[Lemma 2.13]{NGen}: consider the trigonometric version of the shuffle algebra
\begin{equation}
\label{eqn:tilde v}
\widetilde{\CV}^+_{\text{loc}} = \bigoplus_{\bn = (n_i)_{i \in I} \in \nn} \field[w_{i1}^{\pm 1},\dots,w_{in_i}^{\pm 1}]^{\sym}_{i \in I}
\end{equation}
where we interpret $w_{ia}$ as $e^{z_{ia}}$ and define the shuffle product on $\widetilde{\CV}^+_{\text{loc}}$ using the trigonometric zeta function
$$
\widetilde{\zeta}_{ij}(y) = \left(\frac {1-ye^{-\hbar}}{1-y} \right)^{\delta_{ij}} \prod_{(\alpha : i \rightarrow j) \in \bar{\edge}} (1-ye^{u_\alpha})
$$
Consider any $(\bn,g)$ as in \eqref{eqn:special}; we may assume that $g$ is a color-symmetric polynomial only in the $z_{ia}$'s (with no dependence on $\hbar$ or $u_\alpha$). Then the shuffle element
$$
\widetilde{e}_{\bn,g} = g(1-w_{i1},\dots,1-w_{in_i})_{i \in I} \prod_{i \in I} \prod_{1 \leq a , b \leq n_i} \left(1- \frac {w_{ia}}{w_{ib}e^\hbar}\right) 
$$
lies in the subalgebra $\widetilde{\CS}^+_{\text{loc}} \subset \widetilde{\CV}^+_{\text{loc}}$ consisting of Laurent polynomials that satisfy the natural trigonometric analogues of the divisibility conditions \eqref{eqn:divisibility loc}. By \cite[Theorem 1.2]{Wheel}, the subalgebra $\widetilde{\CS}^+_{\text{loc}}$ is generated by the shuffle elements in one variable $w_{i1}^k$, as $i$ runs over $I$ and $k$ runs over $\BZ$. Therefore, we have the following identity
\begin{equation}
\label{eqn:eqn}
\widetilde{e}_{\bn,g} = \text{linear combination of Sym} \left[w_{i_1\bullet_1}^{k_1} \dots w_{i_n\bullet_n}^{k_n} \prod_{1 \leq a < b \leq n} \widetilde{\zeta}_{i_ai_b} \left(\frac {w_{i_a\bullet_a}}{w_{i_b\bullet_b}} \right) \right]
\end{equation}
for various $i_1,\dots,i_n \in I$, $k_1,\dots,k_n \in \BZ$, where $\bullet_1,\dots,\bullet_n$ are the minimal positive integers such that $\bullet_a < \bullet_b$ whenever $a<b$ and $i_a = i_b$. If we write $w_{i\bullet} = e^{z_{i\bullet}}$ in \eqref{eqn:eqn} and extract the coefficient of smallest degree in $z_{ia},\hbar,u_\alpha$, then we obtain
$$
e_{\bn,g} = \text{linear combination of Sym} \left[z_{i_1\bullet_1}^{\ell_1} \dots z_{i_n\bullet_n}^{\ell_n} \prod_{1 \leq a < b \leq n} \zeta_{i_ai_b} \left(z_{i_a\bullet_a} - z_{i_b\bullet_b} \right) \right]
$$
for various $i_1,\dots,i_n \in I$, $\ell_1,\dots,\ell_n \geq 0$, precisely what we needed to prove. \end{proof} 

\medskip 

\begin{remark} We note an imprecision in the proof of Theorem \ref{thm:shuffle loc}: the exponentials $e^{\hbar}, e^{u_\alpha}$ do not make sense for arbitrary elements $\hbar,u_\alpha$ in an arbitrary field, although they do make sense when these symbols arise from a ring homomorphism \eqref{eqn:geometric parameters}. Thus, our Theorem \ref{thm:shuffle loc} should only be assumed to be proved in the more restrictive latter case, which applies to the main geometric situation considered in Section \ref{sec:coha}.

\end{remark}

\medskip 

\subsection{Specialization}
\label{sub:specialization}

Throughout the present Subsection, we assume that there exists a ring $\bar{\ring}$ and a homomorphism
$$
\ring \rightarrow \bar{\ring} \quad \text{that sends} \quad \hbar \mapsto 0.
$$
This allows us to consider the $\bar{\ring}$-algebra
$$
\bar{\CS}^+ = \CS^+ \otimes_{\ring} \bar{\ring}
$$
Note that just like $\CS$, the algebra $\bar{\CS}^+$ is graded by $\nn$ via the number of variables.

\medskip 

\begin{proposition}
\label{prop:spec is comm}

The ring $\bar{\CS}^+$ is supercommutative, i.e.
$$
\bar{E} * \bar{E'} = (-1)^{(\bn,\bn')'}\bar{E}' * \bar{E}
$$
for all $\bar{E} \in \bar{\CS}_{\bn}$ and $\bar{E}' \in \bar{\CS}_{\bn'}$, see \eqref{eqn:modified euler form} for the definition of the bilinear form.

\end{proposition}

\medskip 

\begin{proof} It is enough to prove that for all $E(z_{i1},\dots,z_{in_i})_{i \in I} \in \CS_{\bn}$ and $E'(z_{i1},\dots,z_{in_i'})_{i \in I} \in \CS_{\bn'}$, we have
\begin{equation}
\label{eqn:hbar comm}
(-1)^{\sum_{(\alpha:i\rightarrow j) \in \edge} -n_jn'_i} E*E' - (-1)^{\sum_{(\alpha:i\rightarrow j) \in \edge} n_in'_j} E'*E = \hbar \tilde{E}
\end{equation}
where $\tilde{E}$ satisfies the divisibility properties in Definition \ref{def:shuffle int}. Explicitly, we have 
$$
\text{LHS of \eqref{eqn:hbar comm}} = \text{symmetrization of }A-B
$$
where
\begin{align*}
&A=E(z_{i1},\dots,z_{in_i})_{i \in I} E'(z_{i,n_i+1},\dots,z_{i,n_i+n_i'})_{i \in I}  \prod_{i,j \in I} \prod_{a=1}^{n_i} \prod_{b={n_j+1}}^{n_j+n_j'} (-1)^{|\alpha:j\rightarrow i|} \zeta_{ij} \left( z_{ia} - z_{jb}  \right) \\  
&B=E(z_{i1},\dots,z_{in_i})_{i \in I} E'(z_{i,n_i+1},\dots,z_{i,n_i+n_i'})_{i \in I}  \prod_{i,j \in I} \prod_{a=1}^{n_i} \prod_{b={n_j+1}}^{n_j+n_j'}  (-1)^{|\alpha:i\rightarrow j|}  \zeta_{ji} \left( z_{jb} - z_{ia}  \right)   
\end{align*} 
We now consider any $I$-composition
$$
Q = \left\{n_i+n_i' = k_i^{(1)} + \dots + k_i^{(d_i)} \right\}_{i \in I}
$$
and specialize the symmetrizations of $A$ and $B$ at
$$
\Big\{z_{i1},\dots,z_{i,n_i+n_i'}\Big\}_{i \in I} = \left\{x_{ia},x_{ia}+\hbar,\dots,x_{ia}+ (k_i^{(a)}-1)\hbar \right\}_{i \in I, a \in \{1,\dots,d_i\}}
$$
Because $\zeta_{ii}(\hbar) = 0$, the only non-zero specializations of $A$ are those for which
\begin{align*} 
&x_{ia},\dots,x_{ia}+ (\ell_i^{(a)}-1)\hbar \qquad \qquad \text{are plugged into the variables of }E \\
&x_{ia}+ \ell_i^{(a)}\hbar,\dots,x_{ia} + (k_i^{(a)}-1)\hbar \ \text{ are plugged into the variables of }E'
\end{align*}
and the only non-zero specializations of $B$ are those for which
\begin{align*} 
&x_{ia}+(k_i^{(a)}-\ell_i^{(a)})\hbar,\dots,x_{ia}+ (k_i^{(a)}-1)\hbar \quad \text{are plugged into the variables of }E \\
&x_{ia},\dots,x_{ia}+ (k_i^{(a)}-\ell_i^{(a)}-1)\hbar \qquad \qquad \quad \text{are plugged into the variables of }E'
\end{align*}
for some henceforth fixed $\{0 \leq \ell_i^{(a)} \leq k_i^{(a)}\}_{i,a}$. In the following discussion, we fix the $\ell_i^{(a)}$ and consider the $I$-compositions
\begin{align*}
&P = \left\{n_i = \ell_i^{(1)} + \dots + \ell_i^{(d_i)} \right\}_{i \in I} \\
&P' = \left\{n_i' = (k_i^{(1)} - \ell_i^{(1)}) + \dots + (k_i^{(d_i)} - \ell_i^{(d_i)}) \right\}_{i \in I}
\end{align*}
By the assumption that $E,E' \in \CS^+$, we have
\begin{multline}
\label{eqn:1}
\text{Spec}_P(E) = G(x_{i1},\dots,x_{id_i})_{i \in I} \prod_{i \in I} \left(\hbar^{\ell_i^{(a)}} \prod_{a=1}^{d_i} \ell_i^{(a)}!\right) \\ \prod_{(\alpha : i \rightarrow j) \in \edge} \mathop{\prod_{1 \leq a \leq d_i}}_{1 \leq b \leq d_j} \prod_{c \in \BZ} \left(x_{jb} - x_{ia} + c\hbar - u_\alpha\right)^{\chi_{\ell_i^{(a)},\ell_j^{(b)}}(c)} 
\end{multline}
\begin{multline}
\label{eqn:2}
\text{Spec}_{P'}(E') = G'(x_{i1},\dots,x_{id_i})_{i \in I} \prod_{i \in I} \left(\hbar^{k_i^{(a)} - \ell_i^{(a)}} \prod_{a=1}^{d_i} (k_i^{(a)} - \ell_i^{(a)})!\right) \\ \prod_{(\alpha : i \rightarrow j) \in \edge} \mathop{\prod_{1 \leq a \leq d_i}}_{1 \leq b \leq d_j} \prod_{c \in \BZ} \left(x_{jb} - x_{ia} + c\hbar - u_\alpha\right)^{\chi_{k_i^{(a)} -\ell_i^{(a)},k_j^{(b)} - \ell_j^{(b)}}(c)} 
\end{multline}
for polynomials $G,G'$ with coefficients in $\ring$. Therefore, we have
\begin{equation}
\label{eqn:a}
\text{Spec}_Q(A) = G(x_{ia})_{i\in I, a\in \{1,\dots,d_i\}} G'(x_{ia}+\ell_i^{(a)}\hbar)_{i\in I, a\in \{1,\dots,d_i\}}
\end{equation}
$$
\prod_{i \in I} \left[ \prod_{a=1}^{d_i} \left( \hbar^{\ell_i^{(a)}}\ell_i^{(a)}!\right)  \prod_{a=1}^{d_i} \left(\hbar^{k_i^{(a)}-\ell_i^{(a)}}(k_i^{(a)} - \ell_i^{(a)})!\right) \prod_{a=1}^{d_i} \prod_{r=0}^{(\ell_i^{(a)}-1)} \prod_{s=\ell_i^{(a)}}^{(k_i^{(a)}-1)} \frac {r-s-1}{r-s} \right] \mathop{\prod_{1 \leq a \leq d_i}^{(\alpha : i \rightarrow j) \in \edge}}_{1 \leq b \leq d_j} 
$$
$$
\left\{\prod_{c \in \BZ} \left[ \left(x_{jb} - x_{ia} + c\hbar - u_\alpha\right)^{\chi_{\ell_i^{(a)},\ell_j^{(b)}}(c)} \left(x_{jb} - x_{ia} + c\hbar - u_\alpha\right)^{\chi_{k_i^{(a)} -\ell_i^{(a)},k_j^{(b)} - \ell_j^{(b)}}(c+\ell_i^{(a)}-\ell_j^{(b)})} \right] \right.
$$
$$
\left. \prod_{r=0}^{\ell_i^{(a)}-1} \prod_{s=\ell_j^{(b)}}^{k_j^{(b)}-1}  (x_{jb}-x_{ia}+(s-r)\hbar - u_\alpha) \prod_{r=\ell_i^{(a)}}^{k_i^{(a)}-1} \prod_{s=0}^{\ell_j^{(b)}-1} (x_{jb}-x_{ia}+(s+1-r)\hbar - u_\alpha) \right\}
$$
and
\begin{equation}
\label{eqn:b}
\text{Spec}_Q(B) = G(x_{ia}+(k_i^{(a)}-\ell_i^{(a)})\hbar)_{i\in I, a\in \{1,\dots,d_i\}} G'(x_{ia})_{i\in I, a\in \{1,\dots,d_i\}}
\end{equation}
$$
\prod_{i \in I} \left[\prod_{a=1}^{d_i} \left( \hbar^{\ell_i^{(a)}}\ell_i^{(a)}!\right)  \prod_{a=1}^{d_i} \left(\hbar^{k_i^{(a)}-\ell_i^{(a)}}(k_i^{(a)} - \ell_i^{(a)})!\right)\prod_{a=1}^{d_i} \prod_{r=k_i^{(a)}-\ell_i^{(a)}}^{(k_i^{(a)}-1)} \prod_{s=0}^{(k_i^{(a)}-\ell_i^{(a)}-1)} \frac {s-r-1}{s-r} \right] \mathop{\prod_{1 \leq a \leq d_i}^{(\alpha : i \rightarrow j) \in \edge}}_{1 \leq b \leq d_j} 
$$
$$
\left\{ \prod_{c \in \BZ} \left[ \left(x_{jb} - x_{ia} + c\hbar - u_\alpha\right)^{\chi_{\ell_i^{(a)},\ell_j^{(b)}}(c+k_i^{(a)}-\ell_i^{(a)}-k_j^{(b)}+\ell_j^{(b)})}  \left(x_{jb} - x_{ia} + c\hbar - u_\alpha\right)^{\chi_{k_i^{(a)} -\ell_i^{(a)},k_j^{(b)} - \ell_j^{(b)}}(c)} \right] \right.
$$
$$
\left. \prod_{r=0}^{k_i^{(a)}-\ell_i^{(a)}-1} \prod_{s=k_j^{(b)}-\ell_j^{(b)}}^{k_j^{(b)}-1} (x_{jb}-x_{ia}+(s-r)\hbar - u_\alpha) \prod_{r=k_i^{(a)}-\ell_i^{(a)}}^{k_i^{(a)}-1} \prod_{s=0}^{k_j^{(b)}-\ell_j^{(b)}-1} (x_{jb}-x_{ia}+(s+1-r)\hbar - u_\alpha)  \right\}
$$
Comparing the formulas above for $\text{Spec}_Q(A)$ and $\text{Spec}_Q(B)$, we note the following

\medskip 

\begin{itemize}[leftmargin=*]

\item the first rows of \eqref{eqn:a} and \eqref{eqn:b} are equal to each other modulo $\hbar$

\medskip 

\item the second rows of \eqref{eqn:a} and \eqref{eqn:b} are both equal to $\prod_{i \in I} \prod_{a=1}^{d_i} \left( \hbar^{k_i^{(a)}}k_i^{(a)}! \right)$

\medskip 

\item the third and fourth rows of \eqref{eqn:a} multiply out to $\prod_{c \in \BZ} \left(x_{jb} - x_{ia} + c\hbar - u_\alpha\right)^{\chi_{k_i^{(a)},k_j^{(b)}}(c)}$ times the quantity
\begin{multline}
\label{eqn:aa}
\prod_{\max(0,k_j^{(b)}-k_i^{(a)})<c\leq \ell_j^{(b)}-\ell_i^{(a)}} (x_{jb}-x_{ia} + c\hbar - u_\alpha) \\ \prod_{\min(0,k_j^{(b)}-k_i^{(a)})\geq c > \ell_j^{(b)}-\ell_i^{(a)}} (x_{jb}-x_{ia} + c\hbar - u_\alpha)
\end{multline}

\medskip 

\item the third and fourth rows of \eqref{eqn:b} multiply out to $\prod_{c \in \BZ} \left(x_{jb} - x_{ia} + c\hbar - u_\alpha\right)^{\chi_{k_i^{(a)},k_j^{(b)}}(c)}$ times the quantity
\begin{multline}
\label{eqn:bb}
\prod_{\max(0,k_j^{(b)}-k_i^{(a)})<c\leq k_j^{(b)}-\ell_j^{(b)}-k_i^{(a)}+\ell_i^{(a)}} (x_{jb}-x_{ia} + c\hbar - u_\alpha) \\ \prod_{\min(0,k_j^{(b)}-k_i^{(a)})\geq c > k_j^{(b)}-\ell_j^{(b)}-k_i^{(a)}+\ell_i^{(a)}} (x_{jb}-x_{ia} + c\hbar - u_\alpha)
\end{multline}

\end{itemize}

\medskip 

\noindent (the statements in the latter two bullets are due to combinatorial identities that were proved at the end of the proof of \cite[Proposition 5.1]{JNCoulomb}). While formulas \eqref{eqn:aa} and \eqref{eqn:bb} are not equal on the nose, they are equal modulo $\hbar$. Therefore, the four bullets above imply that $\text{Spec}_Q(A-B)$ is a multiple of 
\begin{equation}
\label{eqn:final}
\prod_{i \in I} \prod_{a=1}^{d_i} \left( \hbar^{k_i^{(a)}}k_i^{(a)}! \right) \prod_{c \in \BZ} \left(x_{jb} - x_{ia} + c\hbar - u_\alpha\right)^{\chi_{k_i^{(a)},k_j^{(b)}}(c)}
\end{equation}
times a polynomial in $\ring[x_{i1},\dots,x_{id_i}]_{i \in I}$ which is a multiple of $\hbar$. Plugging this into equation \eqref{eqn:hbar comm} implies that $\text{Spec}_Q(\tilde{E})$ is a multiple of \eqref{eqn:final}. According to Definition \ref{def:shuffle int}, this precisely establishes that $\tilde{E} \in \CS^+$, as required. \end{proof}

\medskip

\subsection{The negative shuffle algebras}

The (big) shuffle algebras with superscript $-$ will refer to the opposites of the ones introduced in the previous subsections:
\begin{align*}
&\CV^- = \CV^{+,\text{op}}, \qquad \CV^-_{\text{loc}} = \CV^{+,\text{op}}_{\text{loc}} \\
&\CS^- = \CS^{+,\text{op}}, \qquad \CS^-_{\text{loc}} = \CS^{+,\text{op}}_{\text{loc}}
\end{align*}
We will denote the analogous elements to \eqref{eqn:spherical} as $f_{i,k} \in \CV^-_{\text{loc}}$ and those of \eqref{eqn:special} as
\begin{equation}
\label{eqn:special transpose}
f_{\bn,g} \in \CV^-
\end{equation}
They are given by the exact same formulas as their $e$ counterparts, but are viewed as elements in the opposite shuffle algebras $\CV^-_{\text{loc}}$ and $\CV^-$, respectively. We define the horizontal grading on shuffle algebras (valued in $\zz$) as
\begin{equation}
\label{eqn:deg h}
\hdeg E = \bn \quad \text{and} \quad \hdeg F = -\bn
\end{equation}
for any $E(z_{i1},\dots,z_{in_i})_{i \in I} \in \CV^+_{\text{loc}}$ and $F(z_{i1},\dots,z_{in_i})_{i \in I} \in \CV^-_{\text{loc}}$. The following proposition is straightforward, and we leave it as an exercise to the reader.

\medskip 

\begin{proposition}
\label{prop:o}

Let us consider the following rational functions in place of \eqref{eqn:zeta}
\begin{equation}
\label{eqn:ozeta}
\ozeta_{ij}(x) = \left( \frac x{x+\hbar} \right)^{\delta_{ij}} \prod_{(\alpha : j \rightarrow i) \in \edge} \left(- \frac {x+\hbar-u_\alpha}{x-u_\alpha} \right)
\end{equation}
Because $\frac {\ozeta_{ij}(x)}{\ozeta_{ji}(-x)} = \frac {\zeta_{ij}(x)}{\zeta_{ji}(-x)}$, the map
\begin{equation}
\label{eqn:xi}
\Xi : \field [z_{i1},\dots,z_{in_i}]_{i \in I}^{\emph{sym}} \longrightarrow \frac {\field [z_{i1},\dots,z_{in_i}]_{i \in I}^{\emph{sym}} \cdot \prod^{i \in I}_{1\leq a \neq b \leq n_i} (z_{ia} - z_{ib})}{\prod^{i \in I}_{1\leq a \neq b \leq n_i} (z_{ia} - z_{ib} - \hbar)\prod^{(\alpha : i \rightarrow j)\in \edge}_{a \leq n_i, b \leq n_j} (z_{ia} + u_{\alpha} - z_{jb})}
\end{equation}
of multiplication by
\begin{equation}
\label{eqn:division}
\frac{\prod^{i \in I}_{1 \leq a \neq b \leq n_i} \left( z_{ib}-z_{ia} \right)}{\prod^{i \in I}_{1 \leq a \neq b \leq n_i} \left(z_{ib} - z_{ia} - \hbar\right)  \prod^{(\alpha : i \rightarrow j) \in \edge}_{1 \leq a \leq n_i, 1\leq b \leq n_j,(i,a) \neq (j,b)} \left(z_{jb}-z_{ia}-u_\alpha \right)}
\end{equation}
intertwines the shuffle product $*$ of \eqref{eqn:mult} in the LHS with the shuffle product $\bar{*}$ in the RHS, the latter being defined using the functions \eqref{eqn:ozeta} instead of \eqref{eqn:zeta}.

\end{proposition}

\medskip 

\subsection{The shifted double}
\label{sub:double}

For any $\br = (r_i)_{i \in I} \in \zz$, we define the (extended) shifted double shuffle algebra as
\begin{equation}
\label{eqn:double loc}
\CS^{\br}_{\text{loc}} = \CS^+_{\text{loc}} \otimes \field [p_{i,k}, c_{i, k} ]_{i \in I, k \geq 0} \otimes \CS^-_{\text{loc}}
\end{equation}
The commuting elements $p_{i,k}, c_{i,k}$ will be called Cartan elements. The multiplication between the three factors of \eqref{eqn:double loc} is controlled by relations \eqref{eqn:rel double 0}-\eqref{eqn:rel double 2} below 
\begin{equation}
\label{eqn:rel double 0}
c_{i, k} \text{ are central}
\end{equation}
\begin{equation}
\label{eqn:rel double 1}
\begin{split}
&[p_{i,k}, E] = \sum_{a=1}^{n_i} \left( z_{ia}^k - (z_{ia}+\hbar)^k\right) E \\
&[p_{i,k}, F] = \sum_{a=1}^{n_i} \left(  (z_{ia}+\hbar)^k - z_{ia}^k\right) F
\end{split}
\end{equation}
for all $E(z_{i1},\dots,z_{in_i})_{i \in I} \in \CS^+_{\text{loc}}$, $F(z_{i1},\dots,z_{in_i})_{i \in I} \in \CS^-_{\text{loc}}$, $i \in I$ and $k > 0$. Finally, we impose for all $i,j \in I$ and $k,\ell \in \BZ$ the relation
\begin{equation}
\label{eqn:rel double 2}
\Big[e_{i,k}, f_{j,\ell} \Big] = \delta_{ij} \cdot \frac {\xi_{i,k+\ell}}{\gamma_i} 
\end{equation} 
where $\gamma_i = \prod_{\alpha : i \rightarrow i} \left[ u_\alpha(\hbar-u_\alpha)\right]$ and
$$
\xi_i(x) = \sum_{k=r_i}^\infty \frac {\xi_{i,k}}{x^k} = x^{-r_i} \exp \left[\sum_{k=1}^{\infty} \frac 1{kx^k} \left(-c_{i,k}+p_{i,k}+\sum_{\ell=0}^k p_{i,\ell} (-\hbar)^{k-\ell}{k \choose \ell} \right. \right. 
$$
\begin{equation}
\label{eqn:def xi}
\left. \left. - \sum_{\alpha:i \rightarrow j} \sum_{\ell=0}^{k} p_{j,\ell} (-u_\alpha)^{k-\ell} {k \choose \ell} - \sum_{\alpha:j \rightarrow i} \sum_{\ell=0}^{k} p_{j,\ell} (u_\alpha-\hbar)^{k-\ell} {k \choose \ell} \right) \right]
\end{equation}
(we make the convention that $\xi_{i,k} = 0$ for $k < r_i$). As a consequence of Theorem \ref{thm:shuffle loc}, formulas \eqref{eqn:rel double 0}-\eqref{eqn:rel double 2} are sufficient to completely determine the algebra structure on \eqref{eqn:double loc}: for any $E \in \CS^+_{\text{loc}}$ and $F \in \CS^-_{\text{loc}}$, one evaluates $FE$ by writing 
\begin{align*} 
&E = \text{linear combination of }e_{i_1,k_1} \dots e_{i_m,k_m} \\
&F = \text{linear combination of }f_{j_1,\ell_1}\dots f_{j_n,\ell_n}
\end{align*} 
and then uses \eqref{eqn:rel double 0}-\eqref{eqn:rel double 2} to reorder all the $e$'s to the left of all the $p,c$'s, which in turn should be to the left of all the $f$'s. We define 
\begin{equation}
\label{eqn:integral version}
\CS^{\br} \subset \CS^{\br}_{\text{loc}}
\end{equation}
to be the $\ring$-subalgebra generated by

\medskip 

\begin{itemize}[leftmargin=*]

\item the integral shuffle algebras $\CS^\pm \subset \CS^\pm_{\text{loc}}$, and 

\medskip 

\item the subalgebra $\ring [p_{i,k}, c_{i,k}]_{i \in I, k > 0}$ of Cartan elements.

\end{itemize}

\medskip 

\noindent We do not necessarily expect a triangular decomposition $\CS \neq \CS^+ \otimes (\text{Cartan}) \otimes \CS^-$, especially if the quiver $Q$ has loops.

\medskip

\subsection{Difference operators}
\label{sub:bfn}

Consider any $\bd = (d_i)_{i \in I} \in \nn$ and parameters
\begin{equation}
\label{eqn:framing}
\underbrace{\sigma_{i1},\dots,\sigma_{ik_i}}_{\text{corresponding to arrows }i \rightarrow \square} \quad \text{and} \quad \underbrace{\tau_{i1},\dots,\tau_{i\ell_i}}_{\text{corresponding to arrows } \square \rightarrow i}
\end{equation}
Let $\bk = (k_i)_{i \in I}$ and $\bell = (\ell_i)_{i \in I}$. Consider the $\field$-algebra of difference operators
\begin{equation}
\label{eqn:difference operators}
\DD_{\bd|\bk,\bell} = \frac {\field(w_{ia})_{i \in I, a \in \{1,\dots,d_i\}} \otimes \field[D^{\pm 1}_{ia}]_{i \in I, a \in \{1,\dots,d_i\}}}{D_{ia} w_{jb} = (w_{jb}+\hbar\delta_{ij}\delta_{ab})D_{ia}}[\sigma_{i1},\dots,\sigma_{ik_i}, \tau_{i1},\dots,\tau_{i\ell_i}]_{i \in I}
\end{equation}
We recall the following construction of Finkelberg-Tsymbaliuk (\cite{FT}, type $A$), Frassek-Tsymbaliuk (\cite{FrT}, classical types) and Tsymbaliuk (\cite{Tsymbaliuk}, general type) that provides a shuffle algebra incarnation of the work \cite{GKLO} on Yangians and difference operators. By analogy with \cite[Theorem B.17]{FrT}, we have the following algebra homomorphisms
\begin{equation}
\label{eqn:tsymbaliuk}
\Ts : \CV^\pm_{\text{loc}} \rightarrow \DD_{\bd|\bk,\bell}
\end{equation}
defined for any $E(z_{i1},\dots,z_{in_i})_{i \in I} \in \CV^+_{\text{loc}}$ and $F(z_{i1},\dots,z_{in_i})_{i \in I} \in \CV^-_{\text{loc}}$ by
\begin{equation}
\label{eqn:action e}
E \mapsto \hbar^{-|\bn|} \sum^{I-\text{compositions }P}_{\{n_i = n_i^{(1)}+ \dots + n_i^{(d_i)}\}_{i \in I}} \text{Res}^+_P(\Xi(E))
\end{equation} 
$$
\prod^{(i,a)}_{0 \leq c < n_i^{(a)}} 
\frac {\prod^{\alpha : i \rightarrow j}_{(j,b) \neq (i,a)} (w_{ia} + c \hbar - w_{jb} + u_\alpha ) \prod_{s=1}^{k_i} (w_{ia} + c \hbar - \sigma_{is})}{\prod_{b \neq a} (w_{ia} + c \hbar - w_{ib})} \prod_{(i,a)} D_{ia}^{n_i^{(a)}}
$$
\begin{equation}
\label{eqn:action f}
F \mapsto \hbar^{-|\bn|} \sum^{I-\text{compositions }P}_{\{n_i = n_i^{(1)}+ \dots + n_i^{(d_i)}\}_{i \in I}} \text{Res}^-_P(\Xi(F))
\end{equation} 
$$
\prod^{(i,a)}_{0 \leq c < n_i^{(a)}} 
\frac {\prod^{\alpha : j \rightarrow i}_{(j,b) \neq (i,a)} (w_{ia} + c\hbar - w_{jb} - u_\alpha) \prod_{t=1}^{\ell_i} (w_{ia} + c\hbar - \tau_{it})}{\prod_{b \neq a} (w_{ia} + c\hbar - w_{ia})} \prod_{(i,a)} D_{ia}^{-n_i^{(a)}}
$$
In the formulas above, we recall that $\Xi(E)$ and $\Xi(F)$ of \eqref{eqn:xi} are rational functions in $z_{ia}$ with at most simple poles at $z_{ia} = z_{ib} + \hbar$. For such a rational function $G(z_{i1},\dots,z_{in_i})_{i \in I}$ and any $I$-composition $P = \{n_i = n_i^{(1)}+ \dots + n_i^{(d_i)}\}_{i \in I}$, let
\begin{equation}
\label{eqn:residue}
\text{Res}_P(G) = 
\end{equation}
$$
\left( \mathop{\text{Res}}_{z_{i,\nu_i^{(a)}} = z_{i,\nu_i^{(a-1)}+1} + (n_i^{(a)}-1)\hbar} \cdots \mathop{\text{Res}}_{z_{i,\nu_i^{(a-1)}+2} = z_{i,\nu_i^{(a-1)}+1}+\hbar} G \right)_{i \in I, a \in \{1,\dots, d_i\}}
$$
where $\nu_i^{(a)} = n_i^{(1)}+\dots+n_i^{(a)}$ for all $a \in \{1,\dots,d_i\}$. Then in \eqref{eqn:action e}-\eqref{eqn:action f}, one sets
\begin{align*}
&\text{Res}^+_P(G) = \frac {\text{Res}_P(G)|_{z_{i,\nu_i^{(a-1)}+1 = w_{ia}}}}{\prod_{i \in I} \prod_{a=1}^{d_i} \left[ \hbar^{n_i^{(a)}-1}(n_i^{(a)}-1)!  \prod_{\alpha : i \rightarrow i} \frac {u_\alpha^{n_i^{(a)}-1}}{(u_{\alpha}+\hbar)\dots (u_{\alpha}+(n_i^{(a)}-1)\hbar)}\right]} \\
&\text{Res}^-_P(G) = \frac {\text{Res}_P(G)|_{z_{i,\nu_i^{(a-1)}+1 = w_{ia} -n_i^{(a)}\hbar}}}{\prod_{i \in I} \prod_{a=1}^{d_i} \left[ \hbar^{n_i^{(a)}-1}(n_i^{(a)}-1)!  \prod_{\alpha : i \rightarrow i} \frac {u_\alpha^{n_i^{(a)}-1}}{(u_{\alpha}+\hbar)\dots (u_{\alpha}+(n_i^{(a)}-1)\hbar)}\right]}
\end{align*}
In the following result, we write $\bd^\vee - \bk - \bell$ for the $I$-tuple of integers
$$
2d_i -  \sum_{j \in I} \left(d_j \Big|\text{arrows }i \rightarrow j\Big| + d_j \Big|\text{arrows }j \rightarrow i\Big|\right) - k_i - \ell_i
$$
We denote by $\CS^{\bd^\vee-\bk-\bell}_{\text{loc}}$ the shifted double shuffle algebra of \eqref{eqn:double loc}.

\medskip 

\begin{theorem}
\label{thm:tsymbaliuk}

(\cite{FT, FrT, Tsymbaliuk}) Formulas \eqref{eqn:action e}-\eqref{eqn:action f} yield algebra homomorphisms \eqref{eqn:tsymbaliuk}. The restriction of these homomorphisms to the subalgebras $\CS^\pm_{\emph{loc}} \subset \CV^\pm_{\emph{loc}}$, together with the assignments
\begin{equation}
\label{eqn:assignment 1}
p_{i,k} \mapsto w_{i1}^k+\dots + w_{id_i}^k
\end{equation}
\begin{equation}
\label{eqn:assignment 2}
c_{i,k} \mapsto \sum_{s=1}^{k_i} \sigma_{is}^k + \sum_{t=1}^{\ell_i} \tau_{it}^k 
\end{equation}
glue to a morphism 
\begin{equation}
\label{eqn:tsymbaliuk theorem}
\eTs : \CS^{\bd^\vee-\bk-\bell}_{\emph{loc}} \rightarrow \DD_{\bd|\bk,\bell} 
\end{equation}
that we dub the \emph{Finkelberg-Frassek-Tsymbaliuk homomorphism}. 

\end{theorem}

\medskip 

\noindent The assignments for the Cartan elements in Theorem \ref{thm:tsymbaliuk} imply that the generating series $\xi_i(x)$ of \eqref{eqn:def xi} corresponds to
$$
\frac {\prod^{(j,b)}_{\alpha : i \rightarrow j} \left(x + u_\alpha - w_{jb}\right) \prod^{(j,b)}_{\alpha : j \rightarrow i} \left(x + \hbar - w_{jb} - u_\alpha  \right) \prod_{s=1}^{k_i} \left(x - \sigma_{is} \right) \prod_{t=1}^{\ell_i} \left(x - \tau_{it} \right)}{\prod_{i,a} \left(x - w_{ia}\right)\left(x + \hbar - w_{ia}\right)}
$$
viewed as an element of $\DD_{\bd|\bk,\bell}(x)$.

\medskip 

\subsection{Coulomb branches}
\label{sub:coulomb}

The aforementioned connection between double shuffle algebras and difference operators, which we attribute to Finkelberg-Frassek-Tsymbaliuk, is also compatible with a very important connection between Yangians and Coulomb branch algebras (\cite{BFNQuiver, BDG, KWWY}). In more detail, we recall that for any $\bd \in \nn$ and framing parameters \eqref{eqn:framing}, the Coulomb branch $\ring$-algebra $\CH_{\bd|\bk,\bell}$ was defined in \cite{BFN}. We will not review the definition of this algebra, and simply use the following facts in the present paper:

\medskip 

\begin{itemize}[leftmargin=*]

\item there is an injective $\ring$-algebra homomorphism
\begin{equation}
\label{eqn:injective}
\CH_{\bd|\bk,\bell} \hookrightarrow \DD_{\bd|\bk,\bell}
\end{equation}
induced by equivariant localization (\cite{BFN, BFNQuiver, KWWY}).

\medskip 

\item the image of the map \eqref{eqn:injective} is generated over $\ring$ by the so-called  dressed positive and negative fundamental monopole operators
\begin{equation}
\label{eqn:special e}
E_{\bn,g} = \sum_{\{A_i \subset \{1,\dots,d_i\}, |A_i| = n_i\}_{i \in I}} g(w_{ia})_{i \in I, a \in A_i}
\end{equation} 
$$
\frac {\prod_{i,j \in I} \prod^{\alpha : i \rightarrow j}_{a \in A_i, b \notin A_j} \left(w_{ia}  - w_{jb} + u_\alpha \right) \prod_{i \in I} \prod_{a \in A_i} \prod_{s=1}^{k_i} \left(w_{ia}  - \sigma_{is}\right)}{\prod_{i \in I} \prod_{a \in A_i, b \notin A_i} \left(w_{ia} - w_{ib} \right)} \prod_{i \in I}  \prod_{a \in A_i} D_{ia}
$$
\begin{equation}
\label{eqn:special f}
F_{\bn,g} = \sum_{\{A_i \subset \{1,\dots,d_i\}, |A_i| = n_i\}_{i \in I}} g(w_{ia}-\hbar)_{i \in I, a \in A_i}
\end{equation} 
$$
\frac {\prod_{i,j \in I} \prod^{\alpha : j \rightarrow i}_{a \in A_i, b \notin A_j} \left(w_{ia} - w_{jb} - u_\alpha \right) \prod_{i \in I} \prod_{a \in A_i} \prod_{t=1}^{\ell_i} \left(w_{ia} - \tau_{it} \right)}{\prod_{i \in I} \prod_{a \in A_i, b \notin A_i} \left(w_{ia}-w_{ib}\right)} \prod_{i \in I} \prod_{a \in A_i} D_{ia}^{-1}
$$
for $\b0 \leq \bn \leq \bd$, $g \in \ring[z_{i1},\dots,z_{in_i}]^{\sym}_{i \in I}$ (\cite{BFN, DK, FT, SS, Weekes}).   

\end{itemize}

\medskip 

\noindent The following analogue of \cite[Lemma 2.12]{Tsymbaliuk} (see also \cite{JNCoulomb} for notation closer to ours) is key to the connection between shuffle algebras and Coulomb branches.

\medskip 

\begin{proposition}
\label{prop:key}

For the shuffle elements $e_{\bn,g}, f_{\bn,g}$ of \eqref{eqn:special} and \eqref{eqn:special transpose}, we have
$$
\eTs(e_{\bn,g}) = \begin{cases} E_{\bn,g} &\text{if } \bn \leq \bd \\ 0&\text{otherwise} \end{cases}, \qquad \eTs(f_{\bn,g}) = \begin{cases} F_{\bn,g} &\text{if } \bn \leq \bd \\ 0&\text{otherwise} \end{cases},
$$
for any $\bn \in \nn$ and $g \in \ring[z_{i1},\dots,z_{in_i}]^{\emph{sym}}_{i \in I}$.

\end{proposition}

\medskip 

\noindent Combining Proposition \ref{prop:key} with the two bullets preceding it implies that the homomorphism \eqref{eqn:tsymbaliuk theorem} descends to a surjective $\ring$-algebra homomorphism
\begin{equation}
\label{eqn:hom}
\CS^{\bd^\vee - \bk - \bell} \twoheadrightarrow \CH_{\bd|\bk,\bell}
\end{equation}
The domain of \eqref{eqn:hom} is the integral version \eqref{eqn:integral version} and not the localized version \eqref{eqn:double loc}. 

\medskip 

\begin{proof} \emph{of Theorem \ref{thm:intro main}:} Since we identified $\CS^+\cong \kha$ in \eqref{eqn:iso intro} and define the shifted double of the loop-nilpotent CoHA to be
$$
\CS^{\bd^\vee-\bk-\bell}
$$
we declare the sought-for homomorphism \eqref{eqn:main} to be precisely \eqref{eqn:hom}. \end{proof}

\bigskip

\section{Cohomological Hall algebras}
\label{sec:coha}

\medskip

\noindent We state and prove the main geometric results in the present paper:

\medskip 

\begin{itemize}[leftmargin=*]

\item the definition of the loop-nilpotent CoHA $\kha$ (Subsection \ref{sub:coha});

\medskip 

\item dimensional reduction, which allows us to prove the injectivity of the morphism
$$
\iota : \kha \rightarrow \kzero \cong \CV^+
$$
(Subsection \ref{sub:injection});

\medskip 

\item the identification of $\text{Im }\iota$ with $\CS^+$, which implies that the loop-nilpotent CoHA has generators indexed by $\bn \in \nn$ and symmetric polynomials $g$ (Subsection \ref{sub:image}).

\medskip 

\item the definition of the BPS Lie algebra and of its loop-nilpotent version, which we show to be (almost) the same in Subsection \ref{sub:support}. This implies a new formula for the Kac polynomial of $Q$ in Subsection \ref{sub:kac}, when we prove Corollary \ref{cor:kac}.

\end{itemize}

\medskip 

\subsection{Stacks of quiver representations}
\label{sub:quiver}

Recall that $Q$ denotes any quiver with vertex set $I$ and arrow set $\edge$. Given any dimension vector $\bn = (n_i)_{i \in I}\in \nn$, define 
\begin{align*}
\mathrm{Rep}_{\bn}(Q) &:= \prod_{(\alpha : i \rightarrow j )\in \edge} \mathrm{Hom}(\BC^{n_i},\BC^{n_j}) \\ 
\GL_{\bn} &:= \prod_{i \in I} \GL_{n_i}(\BC)\end{align*} 
where the group $\GL_{\bn}$ acts on $\mathrm{Rep}_{\bn}(Q)$ by conjugation. Let 
\[\mathfrak{M}_{\bn}(Q) := \mathrm{Rep}_{\bn}(Q)/\GL_{\bn} \] 
be the moduli stack of $\bn$-dimensional representations of $Q$. Let 
\[\mathcal{M}_{\bn}(Q) := \Spec \left(\BC[\mathrm{Rep}_{\bn}(Q)]^{\GL_{\bn}} \right) \] 
be the coarse moduli space parameterizing semisimple points of the aforementioned stack. The map which takes a quiver representation to the direct sum of successive quotients in its Jordan-H\"older filtration is
\[ 
\JH_{\bn} \colon \mathfrak{M}_{\bn}(Q) \rightarrow \mathcal{M}_{\bn}(Q).
\]
Suppose we have an algebraic torus $T \cong (\BC^*)^r$ endowed with characters
\begin{equation}
\label{eqn:elementary character}
e^{u_\alpha} : T \rightarrow \BC^*
\end{equation}
for all $\alpha \in \edge$. We can think of the exponents as equivariant parameters
\begin{equation}
\label{eqn:equivariant parameter}
u_\alpha \in R= H^T\point = \BZ[\text{Lie}(T)]
\end{equation}
The torus $T$ acts on the space of quiver representations by scaling the arrow $\alpha$ via the character $e^{-u_{\alpha}}$. We then define the quotient stacks
\[ \mathfrak{M}_{\bn}^{T}(Q) := \mathrm{Rep}_{\bn}(Q)/(\GL_{\bn} \times T) \] 
and
\[ \mathcal{M}^{T}_{\bn}(Q) := \mathcal{M}_{\bn}(Q)/T\] 
which are well-defined since the $T$ and $\GL_{\bn}$ actions on $\mathrm{Rep}_{\bn}(Q)$ commute. We have an induced morphism
\begin{equation}
\label{eqn:semisimplification}
\JH^{T}_{\bn} \colon \mathfrak{M}^{T}_{\bn}(Q) \rightarrow \mathcal{M}^{T}_{\bn}(Q). 
\end{equation}
The notions above also apply to the doubled and tripled quivers ($\dQ$ and $\tQ$, respectively) with the same vertex set as $Q$, but arrow set enlarged as follows:

\medskip 

\begin{itemize}[leftmargin=*]

\item the doubled quiver $\dQ$ has an additional arrow $\dalpha:j \rightarrow i$ (on which the torus acts by the character $e^{-u_{\dalpha}}$, where $u_{\dalpha}=\hbar-u_\alpha$) for every arrow $\alpha : i \rightarrow j$ of $Q$. Here 
$$
e^{\hbar} : T \rightarrow \BC^*
$$
denotes some henceforth fixed character that does not depend on $\alpha$;  

\medskip 

\item the tripled quiver $\tQ$ involves adding to $\dQ$ a loop $\omega_i : i \rightarrow i$ (on which the torus $T$ acts by the character $e^{\hbar}$) for all $i \in I$. Its representations are depicted graphically as
\begin{equation}
\label{eqn:representation tilde}
\text{Rep}_{\bn}(\tQ) = \Big\{\underset{\omega_i}{\underset{\circlearrowleft}{\BC}}^{n_i} \xleftrightharpoons[\alpha]{\dalpha} \underset{\omega_j}{\underset{\circlearrowleft}{\BC}}^{n_j}\Big\}
\end{equation}

\end{itemize}

\medskip 

\subsection{Vanishing cycles} 
\label{sub:vanishing cycles}

We recall some basics about vanishing cycles and derived categories, with \cite{davison2020cohomological} and \cite{achar} as references. Let $X$ be a finite type separated scheme over $\BC$. We let $\mathrm{D}_{\mathrm{c}}(X)$ be the derived category of constructible complexes of $\BZ$ modules over $X$. Let 
$$
\ptau^{\leq d} \quad \text{and} \quad \ptau^{\geq d}
$$
denote the perverse truncation functors and let $\Per(X)$ denote the abelian subcategory of $\mathrm{D}_{\mathrm{c}}(X)$ consisting of perverse sheaves. Let $\mathcal{D}^+(X):= D^+(\Per(X))$ denote the bounded below derived category of the abelian category of perverse sheaves. Let $T$ be any torus acting on $X$, with quotient stack $X/T$. Now let $T$ be any torus acting on $X$ with quotient stack $X/T$. The shifted perverse $t$-structure on the $T$-equivariant derived category $\mathrm{D}^{+}_{\mathrm{c}}(X/T)$ is defined by setting 
$$
{}^{\mathfrak{p}}\!\tau^{\leq d} : = {}^{\mathfrak{p}}\!\tau^{\leq d - \dim(T)} \quad \text{and} \quad {}^{\mathfrak{p}}\!\tau^{\geq d} := {}^{\mathfrak{p}}\!\tau^{\geq d - \dim(T)}.
$$
Let $\Per(X/T)$ be the heart with respect to this shifted perverse structure and let $\mathcal{D}^+_c(X/T) := \mathcal{D}^+(\Per(X/T))$.  In particular, if $Y \subset X$ is a smooth connected component, then the shifted constant sheaf 
$$
\BZ^{\vir}_{Y/T}:= \BZ_{Y/T}[\dim Y] 
$$
is in  $\Per(X/T)$ (\cite[Lemma 3.3.12]{achar}). If $X=Y/G$ is a smooth quotient stack with an action of $T$ and $f\colon Y\rightarrow \mathbb{C}$ is a $T$-invariant regular function, then there is a perverse exact functor 
$$
\varphi_f\colon \mathcal{D}^+_{\mathrm{c}}(X/T)\rightarrow \mathcal{D}^+_{\mathrm{c}}(X/T)
$$
called the vanishing cycle functor.  We refer to \cite[Section 8.6]{schapira} for the definition and \cite[Proposition 2.13]{davison2020cohomological} for the essential properties that it satisfies.

\medskip 

The coarse moduli space $\mathcal{M}^T(Q)$ is a monoid with respect to the map
\[ 
\oplus: \mathcal{M}^T(Q) \times \mathcal{M}^T(Q) \rightarrow \mathcal{M}^T(Q)
\] 
of direct sum of semisimple quiver representations. This morphism is known to be finite \cite{svenrein}. It induces a monoidal product on the derived category of constructible sheaves $\mathcal{D}^+_c(\mathcal{M}^T(Q))$ defined by 
\begin{equation}
\label{eqn:box}
\mathcal{F}_1 \boxtimes_{\oplus} \mathcal{F}_2 := \oplus_{*}(\mathcal{F}_1 \boxtimes \mathcal{F}_2) 
\end{equation}
which preserves the abelian category of perverse sheaves $\mathrm{Perv}(\mathcal{M}^T(Q))$.

\medskip 

\subsection{Potentials} 
\label{sub:potential}

A \emph{potential} $W$ on $Q$ is any
$\BC$-linear combination of oriented cycles in $Q$ modulo cyclic permutations. If $W$ is a single cycle, we define for any $\alpha \in \edge$
\[ \frac{\partial W}{\partial \alpha} = \sum_{W = c\alpha c'} c'c \] 
and we extend this definition linearly to any $W$. The \emph{Jacobi algebra} of $Q$ is
$$
\text{Jac}(Q,W) = \BC Q \Big / \left( \frac{\partial W}{\partial \alpha} \right)_{\alpha \in \edge}
$$
where $\BC Q$ is the path algebra of $Q$.
Any potential $W$ induces a regular function 
\[ \Tr_{\bn}(W): \mathrm{Rep}_{\bn}(Q) \rightarrow \BC 
\] 
which is $\GL_{\bn}$ invariant due to the conjugation invariance of the trace. Therefore, we obtain an induced regular function on the moduli stack 
\[ \Tr_{\bn}(W): \mathfrak{M}_{\bn}(Q) \rightarrow \BC 
\] 
Whenever it will be clear from context, we will also use the notation $\Tr_{\bn}(W): \mathcal{M}_{\bn}(Q) \rightarrow \BC$ to denote the induced regular function on the coarse moduli space. If the potential is $T$ invariant (i.e. the sum of the torus weights of the arrows in every constituent cycle of $W$ is 0), then we obtain an induced regular function on the stack
$$
\mathrm{Tr}^T_{\bn}(W): \mathfrak{M}^T_{\bn}(Q) \rightarrow \mathbb{C}
$$ 
and similarly on the coarse moduli space $\mathcal{M}^{T}_{\bn}(Q)$.

\medskip

\begin{example}
\label{ex:cubic}

The most important potential considered in the present paper is the canonical cubic potential
\begin{equation}
\label{eqn:potential}
\tW = \sum_{(\alpha : i \rightarrow j) \in \edge} \Big( \omega_j \alpha \dalpha - \dalpha \alpha \omega_i \Big)
\end{equation}
of a tripled quiver $\tQ$.

\end{example}

\medskip 

\subsection{The loop-nilpotent CoHA}
\label{sub:coha} 

For any quiver $Q$, let  
\begin{equation}
\label{eqn:full}
\kh:= \bigoplus_{\bn \in \mathbb{N}^{I}} H(\mathfrak{M}_{\bn}^T(\tilde{Q}),\varphi_{\mathrm{Tr}(\tilde{W})}\mathbb{Z}^{\vir}_{\mathfrak{M}_{\bn}^T(\tilde{Q})}) 
\end{equation}
be the Kontsevich-Soibelman cohomological Hall algebra \cite{KS} of the tripled quiver with the canonical cubic potential \eqref{eqn:potential}. We refer to \cite{davison2022integrality} for the definition when $T$ is non-trivial. We abbreviate \eqref{eqn:full} as the CoHA. 
\medskip 

Via dimensional reduction, the algebra \eqref{eqn:full} is often identified (\cite[Appendix]{Ren:2015zua} and \cite{Yang_2019}) with the preprojective CoHA $\mathcal{A}^T_{\Pi_Q}$ that was defined for the Jordan quiver in \cite{schiffmann2012cherednik} and for arbitrary quivers in \cite{Yang_2018}. Here, $\Pi_Q$ denotes the moduli stack of representations of the doubled quiver $\dQ$ which satisfy the equality
$$
\sum_{(\alpha: i \rightarrow j) \in \edge} (\dalpha \alpha - \alpha \dalpha) = 0
$$
Our main object of study is the following modification of the CoHA \eqref{eqn:full}, where we impose the closed condition that the loops $\omega_i$ act by nilpotent operators. As we shall see in the proof of theorem \ref{thm: injection}, the two versions of CoHAs are isomorphic after localization (subject to some genericity assumptions on $T$).

\medskip

\begin{definition}
\label{def:loop nilpotent}

Let 
\begin{equation}
\label{eqn:closed embedding}
i_{\omega} : \mathfrak{M}^{T, \omega\emph{-nilp}}(\tilde{Q}) \subset \mathfrak{M}^{T}(\tilde{Q})
\end{equation}
denote the closed embedding of the reduced substack of $\tilde{Q}$-representations on which the loops $\omega_i$ act nilpotently for all $i \in I$. We call 
\begin{equation}
\label{eqn:loop nilpotent}
\ekha := \bigoplus_{\bn \in \mathbb{N}^{I}}  H(\mathfrak{M}^{T, \omega\emph{-nilp}}_{\bn}(\tilde{Q}), i^{!}_{\omega}\varphi_{\Tr(\tilde{W})}\mathbb{Z}^{\vir}_{\mathfrak{M}_{\bn}^T(\tilde{Q})}) 
\end{equation} 
the loop-nilpotent cohomological Hall algebra. The algebra structure on \eqref{eqn:loop nilpotent} is defined by the usual CoHA procedure, since loop-nilpotent representations form a Serre subcategory in the sense of \cite{davison2020cohomological}.

\end{definition}

\medskip 

\noindent Let $Q^+$ denote the quiver with the same vertex set as $Q$ but with added loops $\omega_i$ at each vertex. Finite-dimensional representations of this quiver are depicted as follows
$$
\text{Rep}_{\bn}(Q^+) = \Big\{\underset{\omega_i}{\underset{\circlearrowleft}{\BC}}^{n_i} \xrightarrow{\alpha} \underset{\omega_j}{\underset{\circlearrowleft}{\BC}}^{n_j}\Big\}
$$
for any $\bn = (n_i)_{i\in I} \in \nn$. Let $Z_{Q^{+}}$ be the closed subscheme of $\mathrm{Rep}(Q^{+})$ defined by representations $\rho$ which satisfy the equation
$$
\alpha \omega_i = \omega_j \alpha, \qquad \forall \alpha : i \rightarrow j
$$
Let $ Z^{\omega\text{-nilp}}_{Q^{+}} \subset Z_{Q^{+}}$ be the closed sub-scheme of representations where the linear operators corresponding to the loops $\omega_i$ are nilpotent for all $i \in I$. Let $Z^{\omega\text{-nilp}}_{Q^{+},\bn}$ and $Z_{Q^{+},\bn}$ denote the components of $\bn$-dimensional representations. We define 
\[ 
\mathcal{A}^{T, \omega\text{-nilp}}_{Z_{Q^{+}}} := \bigoplus_{\bn \in \mathbb{N}^I} H^{\mathrm{BM}}_{T \times \mathrm{GL}_{\bn}}(Z^{\omega\text{-nilp}}_{Q^{+},\bn}, \mathbb{Z})
\] 
and 
\[ \mathcal{A}^{T}_{Z_{Q^{+}}} := \bigoplus_{\bn \in \mathbb{N}^I} H^{\mathrm{BM}}_{T \times \mathrm{GL}_{\bn}}(Z_{Q^{+},\bn}, \mathbb{Z})
\]  
where the right-hand sides denote Borel-Moore homology groups.
The dimensional reduction theorem \cite[Theorem A.11]{davisonlocalized}\footnote{As stated, the dimensional reduction theorem of \cite{davisonlocalized} holds for vanishing cycles with coefficients in $\mathbb{Q}$, but the same proof works over $\mathbb{Z}$.}, invoked for the cut given by the arrows $\{\dalpha\}_{\alpha \in \edge}$ of the tripled quiver $\tilde{Q}$, gives isomorphisms of graded abelian groups
\begin{equation}
H(\mathfrak{M}^{T}_{\bn}(\tilde{Q}), \varphi_{\Tr(\tilde{W})}\mathbb{ Z}^{\vir}_{\mathfrak{M}_{\bn}^T(\tilde{Q})} ) \simeq H^{\mathrm{BM}}_{T \times \mathrm{GL}_{\bn}}(Z_{Q^{+},\bn}, \mathbb{Z})
\end{equation}
\begin{equation*}
\Rightarrow \quad \pi_1:\mathcal{A}^T_{\tilde{Q},\tilde{W}} \xrightarrow{\sim} \mathcal{A}^{T}_{Z_{Q^{+}}} 
\end{equation*} 
and 
\begin{equation} \label{eqn:dimensionreductionnilpotent}
H(\mathfrak{M}^{T,\omega\text{-nilp}}_{\bn}(\tilde{Q}), i^{!}_{\omega}\varphi_{\Tr(\tilde{W})}\mathbb{Z}^{\vir}_{\mathfrak{M}_{\bn}^T(\tilde{Q})} ) \simeq H^{\mathrm{BM}}_{T \times \mathrm{GL}_{\bn}}(Z^{\omega\text{-nilp}}_{Q^{+},\bn},\mathbb{Z})
\end{equation}
\begin{equation*}
\Rightarrow \quad \pi_2 \colon \mathcal{A}^{T,\omega\text{-nilp}}_{\tilde{Q},\tilde{W}} \xrightarrow{\sim} \mathcal{A}^{T,\omega\text{-nilp}}_{Z_{Q^{+}}} 
\end{equation*}
The isomorphisms above make $\mathcal{A}^{T}_{Z_{Q^{+}}}$ and $\mathcal{A}^{T,\omega\text{-nilp}}_{Z_{Q^{+}}}$ into associative algebras.

\medskip

\subsection{Injection to the shuffle algebra}
\label{sub:injection}

Since $i_{\omega}$ is a closed embedding, the adjunction and functorial properties of vanishing cycles imply that we have a map
\[ 
(i_{\omega})_{*}(i_{\omega})^{!}\varphi_{\Tr(W)} \mathbb{Z}^{\vir}_{\mathfrak{M}_{\bn}^T(\tilde{Q})} \rightarrow \varphi_{\Tr(W)} \mathbb{Z}^{\vir}_{\mathfrak{M}_{\bn}^T(\tilde{Q})} 
\] 
After taking global sections, we obtain a map 
\begin{equation} \label{eqn: pushforward loop nilpotent}
s: \kha \rightarrow \mathcal{A}^{T}_{\tilde{Q},\tilde{W}}
\end{equation}
which is an algebra homomorphism by the same argument as in \cite[Proposition 6.1.13]{J}. Since 
\[ 
Z^{\omega\text{-nilp}}_{Q^{+}} \subset Z_{Q^{+}} 
\] 
is a closed embedding, this induces an algebra homomorphism 
\[ 
\iota_Z: \mathcal{A}^{T,\omega\text{-nilp}}_{Z_{Q^+}} \rightarrow \mathcal{A}^T_{Z_{Q^+}}. 
\] 

\medskip
 
\begin{proposition} \label{dimensionreduction1}
    The following diagram commutes \[\begin{tikzcd}[cramped]
	{\mathcal{A}^{T, \omega\emph{-nilp}}_{\tilde{Q},\tilde{W}}} & {\mathcal{A}^{T}_{\tilde{Q},\tilde{W}}} \\
	{\mathcal{A}^{T,\omega\emph{-nilp}}_{Z_{Q^+}}} & {\mathcal{A}^T_{Z_{Q^+}}}
	\arrow["s", from=1-1, to=1-2]
	\arrow["\simeq"', from=1-1, to=2-1]
	\arrow["{\simeq }", from=1-2, to=2-2]
	\arrow["{\iota_Z}", from=2-1, to=2-2]
\end{tikzcd}\]
where the vertical isomorphisms are given by dimensional reduction. 
\end{proposition}

\begin{proof} This follows by functoriality of dimensional reduction (Lemma A.4 in \cite{Yang_2019}). \end{proof}

\medskip

We furthermore have an algebra homomorphism  
\begin{equation} \label{eqn:shuffle morphism} 
\iota^{\prime} \colon \mathcal{A}^T_{\tilde{Q},\tilde{W}} \rightarrow \mathcal{A}^T_{\tilde{Q}}
\end{equation}
as defined in \cite{botta2023okounkovs} via the functorial properties of vanishing cycles. Here $\mathcal{A}^T_{\tilde{Q}}$ denote the cohomological Hall algebra of tripled quiver $\tilde{Q}$ without potential, defined in \cite{KS}. We shall be identifying $\mathcal{A}^T_{\tilde{Q}} $ with the shuffle algebra $\mathcal{V}^+$ via the isomorphism of algebras $\mathcal{A}^T_{\tilde{Q}} \simeq \mathcal{V}^+$ proved in \cite[Theorem 2]{KS}. Let  
\[ 
\iota^{\prime}_Z: \mathcal{A}^T_{Z_{Q^+}} \rightarrow \mathcal{A}^T_{Q^+} 
\] 
be the morphism of $\mathbb{Z}$ modules given by push-forward along the closed embedding $Z_{Q^+,\bn} \hookrightarrow \mathrm{Rep}_{\bn}(Q^{+})$ and identifying $H_T^{\mathrm{BM}}(\mathfrak{M}_{\bn}(Q^+),\mathbb{Z})$ with $H^T(\mathfrak{M}_{\bn}(Q^+),\mathbb{Z})$ (note that we are ignoring the cohomological degree shift for now). Similarly to Proposition \ref{dimensionreduction1}, we have the following result.

\medskip 

\begin{proposition} \label{dimensionreduction2}
    The following diagram commutes\[\begin{tikzcd}[cramped]
	{\mathcal{A}^T_{\tilde{Q},\tilde{W}} } & {\mathcal{A}^T_{\tilde{Q}} } \\
	{\mathcal{A}^T_{Z_{Q^{+}}}} & {\mathcal{A}^T_{Q^+}}
	\arrow["{\iota^{\prime}}", from=1-1, to=1-2]
	\arrow["{\simeq }", swap, from=1-1, to=2-1]
	\arrow["\simeq", swap, from=2-2, to=1-2]
	\arrow["{\iota^{\prime}_Z}", from=2-1, to=2-2]
\end{tikzcd}\]
where the left-most vertical isomorphism is dimensional reduction and the right-most isomorphism is the pullback along the map of affine spaces (modulo the same group) $p: \mathfrak{M}^T_{\bn}(\tilde{Q}) \rightarrow \mathfrak{M}^T_{\bn}(Q^+)$ that forgets the arrows $\dalpha$ of the tripled quiver.
\end{proposition}

\medskip

\begin{proof} An analogous statement is proved in Proposition 4.5 of \cite{botta2023okounkovs}, specifically involving dimensional reduction along the projection of affine spaces $p^{\prime}: \mathfrak{M}^T_{\bn}(\tilde{Q}) \rightarrow \mathfrak{M}^T_{\bn}(\dQ)$ that forgets the loops. The situation is perfectly analogous in our situation, where instead of the loops we forget the arrows $\dalpha$. \end{proof}

\medskip

\begin{theorem} \label{thm: injection}
Under the geometric assumption \eqref{eqn:assumption geometric}, the homomorphism 
\[ 
\iota^{\prime} \colon \khrate \rightarrow \mathcal{A}^T_{\tilde{Q},\mathrm{rat}} 
\] 
given by tensoring \eqref{eqn:shuffle morphism} by $\mathbb{Q}$ and the composition
\[ 
\iota: \mathcal{A}^{T,\omega\emph{-nilp}}_{\tilde{Q},\tilde{W}} \rightarrow \mathcal{A}^{T}_{\tilde{Q},\tilde{W}}  \rightarrow \mathcal{A}^T_{\tilde{Q}}
\]
are both injective. We conclude that the morphism \eqref{eqn: pushforward loop nilpotent} is also injective.

\end{theorem}

\medskip 

\begin{proof} The injectivity of $\iota$ follows the same proof as in the $K$-theoretic case treated in \cite[Theorem 2.1]{JNCoulomb}, and we don't repeat the argument. The injectivity of $\iota^\prime$ over $\BQ$ is proven under a slightly stronger assumption in \cite{davison2022integrality}, but our geometric assumptions are enough. This is because for an arbitrary torus $T$, it is proved in \cite[Corollary 9.7]{davison2022integrality} that $\khrat$ is a free $R_{\mathbb{Q}}$-module. Thus the natural map
\begin{equation} 
\label{eqn:injection1}
\khrat \hookrightarrow  \khrat \otimes_{\ring_{\mathbb{Q}}} \field
\end{equation}
(where $\field = \text{Frac}(\ring)$) is injective. However, by the equivariant localization theorem, we have an isomorphism
\begin{equation} 
\label{eqn:isomorphism1}
\khrat \otimes_{\ring_{\mathbb{Q}}} \mathbb{\field} \simeq \kharat \otimes_{\ring_{\mathbb{Q}}} \field
\end{equation}
Since $\iota$ is injective and $\mathcal{A}^T_{\tilde{Q}}$ is a torsion-free $H^{T \times \GL_\bn}(\cdot)$ module, it follows that 
\begin{equation} 
\label{eqn:tf 1}
\kha \text{ is a torsion-free }H^{T \times \GL_\bn}(\cdot)\text{-module}
\end{equation}
By \eqref{eqn:injection1} and \eqref{eqn:isomorphism1}, it therefore follows that 
\begin{equation} 
\label{eqn:tf 2}
\khrat \text{ is also torsion-free over }H^{T \times \mathrm{GL}_{\bn}}(\cdot) \otimes_{\mathbb{Z}} \mathbb{Q}
\end{equation}
By Proposition \ref{dimensionreduction2}, the morphism $\iota^{\prime}$ is the pushforward of the closed embedding $Z_{Q^+} \hookrightarrow \mathrm{Rep}(Q^+)$. But the $T \times \mathrm{GL}_{\bn}$ fixed points of $Z_{Q^+}$ and $\mathrm{Rep}(Q^+)$ are the same, so the localization theorem implies that $\iota^{\prime}$ is an isomorphism after localization. Thus the kernel of morphism $\iota^{\prime}$ is contained inside the torsion submodule of 
$$
\mathcal{A}^T_{Z_{Q^+},\text{rat}} \simeq \khrat,
$$
which is trivial. Hence $\iota'$ is injective. \end{proof}

\medskip 

\subsection{The image of $\iota$}
\label{sub:image}

Let $[\omega=0]_{\bn}$ denote the fundamental class of the closed subscheme of $Z^{\omega\text{-nilp}}_{Q^+,\bn}$ defined by the equations $\{\omega_i=0\}_{i \in I}$. We define
\begin{equation}
\label{eqn:def epsilon}
\varepsilon_{\bn,g} = g \cdot [\omega=0]_{\bn}  \in H^{BM}_{T \times \GL_{\bn}}(Z_{Q^+,\bn}^{\omega\text{-nilp}},\mathbb{Z}) 
\end{equation}
for any $g \in H^{T \times \GL_\bn}(\cdot)$, where $\cdot$ denote the usual action of $H^{T \times GL_{\bn}}(\cdot)$ on the Borel-Moore homology of any $T \times GL_{\bn}$ variety. We may regard 
$$
\varepsilon_{\bn,g} \in \kha
$$
via the isomorphism \eqref{eqn:dimensionreductionnilpotent}. The following is straightforward.

\medskip

\begin{lemma}
For any $\bn \in \nn$ and any $g \in \ring$, we have 
\begin{equation}
\label{eqn:koszul}
\iota(\varepsilon_{\bn,g}) =g(z_{i1},\dots,z_{in_i})_{i \in I} \prod_{i \in I} \prod_{1 \leq a , b \leq n_i} \left(z_{ib}-z_{ia}+\hbar\right) = e_{\bn,g}
\end{equation}
where the latter equality involves the identification $\kzero \cong \CV^+$. 

\end{lemma}

\medskip 

\noindent The following result is the cohomological version of of \cite[Proposition 2.2]{JNCoulomb}.

\medskip 

\begin{proposition}
\label{prop:yu zhao}

Under Assumption \soft of \eqref{eqn:assumption soft}, the injection $\iota$ sends 
$$
\ekha \xrightarrow{\sim} \CS^+
$$

\end{proposition}

\medskip 

\begin{proof} \emph{(sketch):} Recall the injection $\iota : \kha \hookrightarrow \CV^+$. We claim that
$$
\text{Im }\iota \subseteq \CS^+
$$
by adapting to cohomology the $K$-theoretic argument given in \cite[Proposition 2.2]{JNCoulomb} (itself inspired by \cite{Zhao}). However, we also have the opposite inclusion
$$
\text{Im }\iota \supseteq \CS^+
$$
which is due to the fact (Theorem \ref{thm:shuffle int}) that $\CS^+$ is generated by $\iota(\varepsilon_{\bn,g})$. \end{proof} 

\medskip 

\noindent The following result is the geometric counterpart of Proposition \ref{prop:spec is comm}.

\medskip 

\begin{corollary} \label{nilCoHAcommute}
Let $T^{\prime} \subset T$ be a subtorus which acts trivially on the loops $\{\omega_i\}_{i \in I}$ of the tripled quiver. The restricted loop-nilpotent cohomological Hall algebra 
\begin{equation}
\label{eqn:iso cor}
\mathcal{A}^{T^{\prime},\omega\emph{-nilp}}_{\tilde{Q},\tilde{W}} \cong \ekha \otimes_{R} R'
\end{equation}
(let $R = \BZ[\emph{Lie}(T)] \twoheadrightarrow \BZ[\emph{Lie}(T')] = R'$ be given by $\hbar \mapsto 0$) is supercommutative.  
\end{corollary}

\medskip 

\begin{proof} By the argument of \cite[Proposition 6.9]{JNCoulomb}, $\mathcal{A}^{T, \omega\text{-nilp}}_{\tilde{Q},\tilde{W}}$ is a free $R$-module, which implies the isomorphism \eqref{eqn:iso cor}. By Propositions \ref{prop:spec is comm} and \ref{prop:yu zhao}, the $\hbar = 0$ specialization of the $\kha$ is supercommutative, and therefore so is $\mathcal{A}^{T^{\prime},\omega\text{-nilp}}_{\tilde{Q},\tilde{W}}$. \end{proof}

\medskip 

\subsection{Cohomological integrality}
\label{sub:cohomological integrality} 

In this section, we explain the cohomological integrality theorem for tripled quivers with the canonical cubic potential. Since the decomposition theorem fails for sheaves with coefficients in $\mathbb{Z}$, we shall be working with sheaves with coefficients in $\mathbb{Q}$. The constructions in Section \ref{sub:vanishing cycles} generalize to $\mathbb{Q}$-coefficients. We refer to \cite{davison2020cohomological} for proofs of all the subsequent statements. In general, we will use the subscript ``rat'' to indicate tensoring over $\BQ$, e.g.
$$
\kharat = \kha \otimes_{\BZ} \BQ
$$
Recall the semisimplification morphism $\JH^T_{\bn}$ of \eqref{eqn:semisimplification}. Although $\JH^T_{\bn}$ is not proper, it was proved in \cite[Lemma 4.1]{davison2020cohomological} (see \cite{kinjo2024decompositiontheoremgoodmoduli} for a more general statement) that we can still apply the decomposition theorem.
One calls 
\[ 
\mathcal{BPS}_{\tilde{Q},\tilde{W},\bn}^T :=\ptau^{\leq 1}\JH^T_{\bn,*} \varphi_{\Tr(\tilde{W})} \mathbb{Q}^{\vir}_{\mathfrak{M}^T(\tilde{Q})}[1]  \in \mathcal{D}^+_c(\mathcal{M}^T_{\bn}(\tilde{Q}))
\]
the \emph{BPS sheaf}. We write
$$
\mathcal{BPS}_{\tilde{Q},\tilde{W}}^T  = \bigoplus_{\bn \in \nn \backslash \b0} \mathcal{BPS}_{\tilde{Q},\tilde{W},\bn}^T 
$$
The natural transformation $\ptau^{\leq 1} \rightarrow \mathrm{id}$ gives a canonical split morphism 
\begin{equation} 
\label{eqn:inclusionofBPSsheaf}
\mathcal{BPS}^T_{\tilde{Q},\tilde{W},\bn}[-1] \rightarrow  \JH^T_{\bn,*} \varphi_{\Tr(\tilde{W})} \mathbb{Q}^{\vir}_{\mathfrak{M}^T(\tilde{Q})}.
\end{equation}
Let \begin{equation}
(\mathrm{id},\mathrm{det}): \mathfrak{M}^T_{\bn}(\tilde{Q}) \rightarrow \mathfrak{M}^T_{\bn}(\tilde{Q}) \times \mathrm{B}\mathbb{C}^*
\end{equation} be the morphism defined by taking the determinant line bundle (see \cite[Section 3.4]{DW}). Taking pullback gives a morphism of sheaves
\begin{equation} \label{actionofBGM}
\JH^T_{\bn,*} \varphi_{\Tr(\tilde{W})} \mathbb{Q}^{\vir}_{\mathfrak{M}^T_{\bn}(\tilde{Q})} \otimes H(\mathrm{B}\mathbb{C}^*,\mathbb{Q}) \rightarrow \JH^T_{\bn,*} \varphi_{\Tr(\tilde{W})} \mathbb{Q}^{\vir}_{\mathfrak{M}^T_{\bn}(\tilde{Q})}.   
\end{equation}
Tensoring the morphism \eqref{eqn:inclusionofBPSsheaf} with $H(\mathrm{B}\mathbb{C}^*,\mathbb{Q})$ and composing with \eqref{actionofBGM} gives 
\begin{equation} 
\label{eqn:inclusionofBPSusheaf}
\mathcal{BPS}_{\tilde{Q},\tilde{W},\bn}[-1] \otimes H(\mathrm{B}\mathbb{C}^*,\mathbb{Q}) \rightarrow  \JH^T_{\bn} \varphi_{\Tr(\tilde{W})} \mathbb{Q}^{\vir}_{\mathfrak{M}^T_{\bn}(\tilde{Q})} 
\end{equation} 
In \eqref{eqn:box}, we explained that there is a symmetric monoidal product $\boxtimes_{\oplus}$ on the derived category of constructible sheaves $\mathcal{D}^+_c(\mathcal{M}^T(\tilde{Q}))$. This allows us to lift the morphism \eqref{eqn:inclusionofBPSusheaf} to 
\begin{equation} 
\label{eqn:rintegralitytheorem}
\mathrm{Sym}_{\boxtimes_{\oplus}}\left(\bigoplus_{\bn \in \nn \backslash \b0} \mathcal{BPS}^T_{\tilde{Q},\tilde{W},\bn}[-1] \otimes H(\mathrm{B}\mathbb{C}^*,\mathbb{Q})\right)  \rightarrow \JH^T_*\varphi_{\Tr(\tilde{W})} \mathbb{Q}^{\vir}_{\mathfrak{M}^T(\tilde{Q})}
\end{equation}
where $\JH^T = \oplus_{\bn \in \nn} \JH^T_{\bn}$. The following result is called the \emph{relative cohomological integrality theorem}, and it was proved in \cite[Theorem A]{davison2020cohomological}.

\medskip 

\begin{theorem}

The morphism \eqref{eqn:rintegralitytheorem} is an isomorphism. 

\end{theorem}

\medskip

\noindent The space of global sections 
\[ 
\mathfrak{g}^T_{\tilde{Q},\tilde{W}}:= H(\mathcal{M}^T(\tilde{Q}), \mathcal{BPS}_{\tilde{Q},\tilde{W}}[-1]) \hookrightarrow H(\mathcal{M}^T(\tilde{Q}),\JH_{*}\varphi_{\Tr(\tilde{W})} \mathbb{Q}^{\vir}_{\mathfrak{M}^T(\tilde{Q})})= \khrat
\] 
is referred to as the \emph{BPS Lie (super)algebra}. Let us explain this name: it was proved in \cite[Corollary 6.11]{davison2020cohomological} that after an appropriate sign-twist $\chi$, we have
$$
[\tbps,\tbps]_{\chi} \subseteq \tbps
$$
where the Lie bracket in the LHS is the (super)commutator in the CoHA
$\mathcal{A}^T_{\tilde{Q},\tilde{W}}$. Thus, the BPS Lie (super)algebra inherits this structure from the sign-twisted multiplication on the CoHA (see \cite[Subsection 3.4.5]{jindalnegutbps} and Section \ref{sec:examples}) for a review of this well-known phenomenon in our notations). The morphism \eqref{eqn:inclusionofBPSusheaf} yields an inclusion of vector spaces
\begin{equation} 
\label{eqn:what is u}
\mathfrak{g}^T_{\tilde{Q},\tilde{W}}[u] \rightarrow \khrat
\end{equation}
Via the algebra structure of $\mathcal{A}^T_{\tilde{Q},\tilde{W},\text{rat}}$, the map above extends to a morphism of $R_{\mathbb{Q}}$-modules 
\begin{equation} \label{eqn:cohomologicalintegrality}
\mathrm{Sym}_{R_{\mathbb{Q}}}\left( \mathfrak{g}^T_{\tilde{Q},\tilde{W}}[u] \right)  \rightarrow  \khrat
\end{equation} 
The following result is called the \emph{absolute cohomology integrality theorem}, and it was proved in \cite[Theorem C]{davison2020cohomological} in the absence of the torus $T$ (see \cite{DW} and \cite{hennecart2026degenerationscohas2calabiyaucategories} for the analogous statement when $T$ is non-trivial). 

\medskip 

\begin{theorem} \label{thm:nilpotentintegrality}

The morphism \eqref{eqn:cohomologicalintegrality} is an isomorphism of $R_{\mathbb{Q}}$-modules. 

\end{theorem}

\medskip 

\noindent We now restrict the isomorphisms \eqref{eqn:rintegralitytheorem} and \eqref{eqn:cohomologicalintegrality} to the closed subset $\mathfrak{M}^{T,\omega\text{-nilp}}({\tilde{Q}})$ of \eqref{eqn:closed embedding}. Let $\overline{i}_{\omega}$ denote the closed embedding $\mathcal{M}^{\omega\text{-nilp}}(\tilde{Q}) \hookrightarrow \mathcal{M}(\tilde{Q})$ induced on the good moduli spaces. 
Applying $\overline{i}_{\omega}^{!}$ to \eqref{eqn:rintegralitytheorem} gives an isomorphism
\begin{equation} 
\label{eqn:nilpotentrintegrality}
\mathrm{Sym}_{\boxtimes_{\oplus}}\left(\bigoplus_{\bn \in \nn \backslash \b0} \overline{i}_{\omega}^!\mathcal{BPS}_{\tilde{Q},\tilde{W},\bn}[-1] \otimes H(\mathrm{B}\mathbb{C}^*,\mathbb{Q})\right) \simeq \JH^T_{*}i_{\omega}^!\varphi_{\Tr(\tilde{W})} \mathbb{Q}^{\vir}_{\mathfrak{M}^T(\tilde{Q})}
\end{equation}
where we have used the fact that $\overline{i}_{\omega}^!$ commutes with $\JH^T_{*}$. This statement is non-trivial since $\JH^T$ is not proper, but it was proved in \cite{davison2020cohomological} and \cite{kinjo2024decompositiontheoremgoodmoduli} that $\JH^T$ satisfies base change and in particular commutes with $\overline{i}_{\omega}^{!}$.  
We will refer to 
\[ 
\mathcal{BPS}^{T,\omega\text{-nilp}}_{\tilde{Q},\tilde{W},\bn}:= \oi_{\omega}^{!} \mathcal{BPS}^T_{\tilde{Q},\tilde{W},\bn} \in \mathcal{D}^+(\mathcal{M}^{T,\omega\text{-nilp}}_{\bn}(\tilde{Q})) 
\] 
as the \emph{loop-nilpotent BPS sheaf}. Taking global sections of \eqref{eqn:nilpotentrintegrality} gives an isomorphism of cohomologically graded vector spaces
\begin{equation} 
\label{eqn:nilpotentintegrality}
\Sym_{R_{\mathbb{Q}}}(\tbpsw[u]) \simeq \kharat
\end{equation} 
where 
\begin{equation} 
\label{eqn:loop-nilpotent bps}
\tbpsw := H(\CM^{T,\omega\text{-nilp}}(\tQ),\mathcal{BPS}^{T,\omega\text{-nilp}}_{\tilde{Q},\tilde{W}}[-1])
\end{equation}
will be called the \emph{loop-nilpotent BPS Lie (super)algebra}.

\medskip 

\subsection{Support of the BPS sheaf} 
\label{sub:support} 

$\mathcal{BPS}^T_{\tilde{Q},\tilde{W}}$ is a perverse sheaf on $\mathcal{M}^T(\tilde{Q})$ and a priori its support can be arbitrary. However, the following result called the \emph{support lemma} was proved in \cite[Lemma 4.1]{davison2022integrality}.

\medskip 

\begin{lemma}
\label{lem:support}

The BPS sheaf $\mathcal{BPS}^T_{\tilde{Q},\tilde{W}}$ is supported on the locus where the loops $\omega_i$ all act by one and the same scalar. 
\end{lemma}

\medskip 

\noindent If we further impose the condition that the loops $\omega_i$ are nilpotent, then the aforementioned scalar is forced to be 0. This immediately implies the following result.

\medskip 

\begin{proposition}

The loop-nilpotent BPS sheaf $\mathcal{BPS}^{T,\omega\emph{-nilp}}_{\tilde{Q},\tilde{W}}$, as well as the restriction 
$$
\oi^*_{\omega} \mathcal{BPS}^T_{\tilde{Q},\tilde{W}},
$$
are supported on the locus $\{\omega_i=0\}_{i \in I}$.

\end{proposition}

\medskip 

\noindent Consider the action map
\[ 
a:\mathbb{A}^1 \times \mathfrak{M}^T(\tilde{Q}) \rightarrow \mathfrak{M}^T(\tilde{Q}) 
\] 
corresponding to the action $\mathbb{A}^1 \curvearrowright \mathfrak{M}^T(\tilde{Q})$ that adds one and the same scalar to all the loops, i.e.
$$
a\Big(\lambda, (\rho(\alpha), \rho(\dalpha), \rho(\omega_i)) \Big) = \Big(\rho(\alpha), \rho(\dalpha), \rho(\omega_i)+\lambda\cdot \text{Id} \Big)
$$
We have a similar action for the good moduli space instead of the stack, so we obtain the commutative diagram
\[
\begin{tikzcd}[cramped]
	{\mathbb{A}^1 \times \mathfrak{M}^T(\tilde{Q}) } & {\mathfrak{M}^T(\tilde{Q}) } \\
	{\mathbb{A}^1 \times \mathcal{M}^T(\tilde{Q}) } & {\mathcal{M}^T(\tilde{Q}) }
	\arrow["a", from=1-1, to=1-2]
	\arrow["{\mathrm{id} \times \JH^T}"', from=1-1, to=2-1]
	\arrow["\JH^T", from=1-2, to=2-2]
	\arrow["{\overline{a}}"', from=2-1, to=2-2]
\end{tikzcd}
\] 
The following result is also known to experts in various contexts.

\medskip 

\begin{proposition} 
\label{eqn:sheafrelation}
    
There is an isomorphism of sheaves 
\begin{equation} 
\label{eqn:isosheaves}
\overline{a}_{*} \Big( \mathbb{Q}_{\mathbb{A}^1}[2] \boxtimes (\overline{i}_{\omega})_{*}\mathcal{BPS}^{T,\omega\emph{-nilp}}_{\tilde{Q},\tilde{W}} \Big) \simeq \mathcal{BPS}^{T}_{\tilde{Q},\tilde{W}}
\end{equation}
\end{proposition}

\medskip 

\begin{proof} The morphism $\Tr(\tW): \mathcal{M}^T(\tilde{Q}) \rightarrow \mathbb{C}$ is invariant with respect to the $\mathbb{A}^1$ action since 
\[ 
\mathrm{Tr}\left( \sum_{i \in I} (\omega_i + \lambda \cdot \text{id})[\alpha,\dalpha] \right) = \Tr(\tilde{W})+ \lambda \sum_{\alpha \in \edge} \Tr([\alpha,\dalpha]) = \Tr(\tilde{W}) 
\] 
and thus $\varphi_{\mathrm{Tr}(\tilde{W})}\mathbb{Q}^{\vir}_{\mathfrak{M}^T(\tilde{Q})}$ is an $\mathbb{A}^1$-equivariant sheaf. Since 
$$
\mathcal{BPS}^T_{\tilde{Q},\tilde{W}} = \ptau^{\leq 1}(\JH^T_*\varphi_{\mathrm{Tr}(\tilde{W})}\mathbb{Q}^{\vir}_{\mathfrak{M}^T(\tilde{Q})}[1])
$$
and $\JH^T$ is an $a$-invariant functor, it follows that $\mathcal{BPS}^T_{\tilde{Q},\tilde{W}}$ is also $\mathbb{A}^1$-invariant. Let 
$$
\overline{i} : \mathcal{M}^{T,\omega\text{-scalar}}(\tilde{Q}) \hookrightarrow \mathcal{M}^T(\tilde{Q})
$$
denote the closed embedding of the locus where all the loops $\omega_i$ act by one and the same scalar. By the support Lemma \ref{lem:support}, we have
\[ 
\mathcal{BPS}^T_{\tilde{Q},\tilde{W}} \simeq \overline{i}_*\overline{i}^* \mathcal{BPS}^T_{\tilde{Q},\tilde{W}}
\]
Note that $\overline{a}$ restricts to an isomorphism 
\[ 
\mathbb{A}^1 \times \mathcal{M}^{T,\omega\text{-nilp}}(\tilde{Q}) \rightarrow \mathcal{M}^{T, \omega\text{-scalar}}(\tilde{Q}).
\] 
Thus, by $\mathbb{A}^1$-equivariance, it follows that
\[ 
\overline{a}_{*}(\mathbb{Q}_{\mathbb{A}^1} \boxtimes \oi_{\omega}^* \mathcal{BPS}^T_{\tilde{Q},\tilde{W}}) \simeq \overline{i}^*\mathcal{BPS}^T_{\tilde{Q},\tilde{W}}
\] 
which by pushing forward under $\overline{i}_{*}$ implies
\begin{equation} \label{eqn:dualstatement}
\overline{a}_{*}( \mathbb{Q}_{\mathbb{A}^1} \boxtimes (\oi_{\omega})_{*}\oi_{\omega}^*\mathcal{BPS}^T_{\tilde{Q},\tilde{W}}) \simeq \mathcal{BPS}^{T}_{\tilde{Q},\tilde{W}}.   
\end{equation}
It is proved in \cite[Theorem A] {davison2020cohomological} that $\mathcal{BPS}^T_{\tilde{Q},\tilde{W}}$ is Verdier self dual, i.e 
$$
\mathbb{D}(\mathcal{BPS}^T_{\tilde{Q},\tilde{W}}) \simeq \mathcal{BPS}^T_{\tilde{Q},\tilde{W}}.
$$
Thus, applying $\mathbb{D}$ to the isomorphism \eqref{eqn:dualstatement} and using the properties $\mathbb{D}f^* \simeq f^! \mathbb{D}$ and $\mathbb{D}f_{*} \simeq f_! \mathbb{D}$ of Verdier duality, formula \eqref{eqn:isosheaves} follows as the restriction of the morphism $\overline{a}$ to the support of $\mathbb{Q}_{\mathbb{A}^1} \boxtimes (\oi_{\omega})_{*}\oi_{\omega}^*\mathcal{BPS}^T_{\tilde{Q},\tilde{W}}$ is proper (by Proposition \ref{lem:support}). \end{proof}

\medskip

\begin{proposition}
\label{prop:bps}

If $T$ satisfies the geometric genericity assumption \eqref{eqn:assumption geometric}, then the image of $\mathfrak{g}^{T, \omega\emph{--nilp}}_{\tilde{Q},\tilde{W}}$ under the morphism
\[ s: \kharate \hookrightarrow \khrate \] coincides with $\hbar \tbps$. 
\end{proposition}

\medskip 

\begin{proof}
The morphism $s$ is induced by the natural map of sheaves \[ (\oi_{\omega})_{*}\oi_{\omega}^! \JH^T_{*} \varphi_{\Tr(\tilde{W})} \mathbb{Q}^{\vir} \rightarrow \JH^T_{*} \varphi_{\Tr(\tilde{W})} \mathbb{Q}^{\vir}.\] Applying the natural morphism $^{\mathfrak{p}}\tau^{\leq 1} \rightarrow \mathrm{id}$ gives a commuting square of morphisms of sheaves
\[\begin{tikzcd}[cramped]
	{(\oi_{\omega})_{*} \oi_{\omega}^! \JH^T_{*} \varphi_{\Tr(\tilde{W})} \mathbb{Q}^{\vir}} & {\JH^T_{*} \varphi_{\Tr(\tilde{W})} \mathbb{Q}^{\vir}} \\
	{(\oi_{\omega})_{*} \oi_{\omega}^{!}\mathcal{BPS}^T_{\tilde{Q},\tilde{W}}[-1]} & {\mathcal{BPS}^T_{\tilde{Q},\tilde{W}}}[-1]
	\arrow[from=1-1, to=1-2]
	\arrow[from=2-1, to=1-1]
	\arrow[from=2-1, to=2-2]
	\arrow[from=2-2, to=1-2]
\end{tikzcd}\]
which implies that $s$ restricts to a morphism 
\[ 
\tbpsw \rightarrow  \tbps 
\]
given by applying derived global sections to the morphism of sheaves \begin{equation} \label{eqn:morphismofbps}
 (\oi_{\omega})_* \mathcal{BPS}^{T,\omega\text{-nilp}}_{\tilde{Q},\tilde{W}}[-1] = (\oi_{\omega})_*\oi_{\omega}^! \mathcal{BPS}^T_{\tilde{Q},\tilde{W}}[-1] \rightarrow \mathcal{BPS}^T_{\tilde{Q},\tilde{W}}[-1].   
\end{equation}
Consider the commutative diagram
\begin{equation}
\label{eqn:commutative diagram}
\begin{tikzcd}[cramped] 
{\mathcal{M}^{\omega \text{-nilp}}(\tilde{Q})} && {\mathcal{M}(\tilde{Q})} \\
{\{0\} \times \mathcal{M}^{\omega\text{-nilp}}(\tilde{Q})} & {\mathbb{A}^1 \times \mathcal{M}^{\omega\text{-nilp}}(\tilde{Q})} & {\mathbb{A}^1 \times \mathcal{M}(\tilde{Q})}
	\arrow["{\oi_{\omega}}", from=1-1, to=1-3]
	\arrow["\overline{a}_0", from=2-1, to=1-1]
	\arrow["{ r \times \mathrm{id}}"', from=2-1, to=2-2]
	\arrow["{\mathrm{id} \times \oi_{\omega}}"', from=2-2, to=2-3]
	\arrow["\overline{a}"', from=2-3, to=1-3]
\end{tikzcd} 
\end{equation}
where $r$ denotes the inclusion $\{0\} \rightarrow \mathbb{A}^1$. Let 
$$
\mathcal{F} = (\oi_{\omega})_*\oi_{\omega}^! \mathcal{BPS}^T_{\tilde{Q},\tilde{W}}[-1] \in \mathcal{D}^+(\CM(\tQ))
$$
Since the diagram \eqref{eqn:commutative diagram} of spaces commutes, we have the following commutative diagram of (complexes of) sheaves:
\[\begin{tikzcd}[cramped]
	{\mathcal{F}} && {\mathcal{BPS}^T_{\tilde{Q},\tilde{W}}[-1]} \\
	{a_{*}(r_{*}r^! \times \mathrm{id})(\mathbb{Q}_{\mathbb{A}^1}[2] \boxtimes \mathcal{F}}) & {\overline{a}_{*}(\mathrm{id} \times (\oi_{\omega})_{*}\oi_{\omega}^!)(\mathbb{Q}_{\mathbb{A}^1}[2] \boxtimes \mathcal{F})} & {\overline{a}_{*}(\mathbb{Q}_{\mathbb{A}^1}[2] \boxtimes \mathcal{F})} \\
	& {\overline{a}_{*}(\mathbb{Q}_{\mathbb{A}^1}[2] \boxtimes \mathcal{F})}
	\arrow["{\eqref{eqn:morphismofbps}}", from=1-1, to=1-3]
	\arrow["{{\simeq (1) } }", from=1-1, to=2-1]
	\arrow["{(3)}", from=2-1, to=2-2]
	\arrow["{(4)}", from=2-1, to=3-2]
	\arrow["{(5)}", from=2-2, to=2-3]
	\arrow["{\simeq (2)}", from=2-2, to=3-2]
	\arrow["{{\simeq \eqref{eqn:isosheaves}}}", from=2-3, to=1-3]
	\arrow["{=}"{description}, from=3-2, to=2-3]
\end{tikzcd}\]
The morphism labeled $(1)$ is an isomorphism because $r^! \mathbb{Q}_{\mathbb{A}^1}[2] \simeq \mathbb{Q}_{0}$. The morphism labeled $(2)$ is an isomorphism because $\overline{i}_{\omega}$ is a closed embedding and thus $\oi^{!}_{\omega}(\oi_{\omega})_{*} \simeq \mathrm{id}.$ The morphism labeled $(3)$ is just the composition of $(4)$ and the inverse of $(2)$. The morphism $(5)$ is defined by adjunction, but since $(2)$ is an isomorphism, we conclude that $(5)$ is an isomorphism. 
Finally, $(4)$ is given by the morphism of sheaves $r_{*}r^! \mathbb{Q}_{\mathbb{A}^1}[2] \rightarrow \mathbb{Q}_{\mathbb{A}^1}[2]$. Since $r^!\mathbb{Q}_{\mathbb{A}^1}[2]=\mathbb{Q}_{0}$ this morphism is the same as the morphism $r_{*} \mathbb{Q}_{0} \rightarrow \mathbb{Q}_{\mathbb{A}^1}[2]$ given by Gysin pushforward of the inclusion $r:\{0\} \rightarrow \mathbb{A}^1$. Applying cohomology, we conclude that $(4)$ is given by the cup product with $\mathrm{eu}(N)$ where $N$ is the normal bundle of $\{0\} \rightarrow \mathbb{A}^1$. Since $\mathrm{eu}(N)=\hbar$, we conclude that the morphism \eqref{eqn:morphismofbps} is given by multiplication with $\hbar$. Thus, 
$$
s(\mathfrak{g}^{T,\omega\text{-nilp}}_{\tilde{Q},\tilde{W}}) \subseteq \hbar \mathfrak{g}^T_{\tilde{Q},\tilde{W}}.
$$ 
We now claim that above containment is an equality. Since $s$ is injective by Theorem \ref{thm: injection}, it is enough to show that 
\begin{equation}
\label{eqn:precisely prove}
\text{graded dimension} \left( \hbar \mathfrak{g}^T_{\tilde{Q},\tilde{W}} \right) = \text{graded dimension} \left(\mathfrak{g}^{T,\omega\text{-nilp}}_{\tilde{Q},\tilde{W}}\right).
\end{equation}
Applying cohomology to Proposition \ref{eqn:sheafrelation} implies that for any fixed dimension vector $\bn$ and any fixed cohomological degree $d$, we have
$$
H^d(\hbar \mathfrak{g}^T_{\tilde{Q},\tilde{W},\bn}) = H^d(\mathfrak{g}^{T,\omega\text{-nilp}}_{\tilde{Q},\tilde{W},\bn}).
$$ 
Since the vector spaces in the left-hand side of the above equation are finite-dimensional (see \cite[Theorem A]{davison2020cohomological}), formula \eqref{eqn:precisely prove} is proved. \end{proof}

\medskip

\begin{corollary} \label{cor:nilpotentassubalgebra}
If $T$ satisfies the geometric assumption \eqref{eqn:assumption geometric}, then $\kharate$ is the subalgebra of $\khrate$ generated by $\hbar \tbps[u]$ (see \eqref{eqn:what is u}).
\end{corollary}

\medskip

\begin{proof} Since $T$ satisfy assumption \eqref{eqn:assumption geometric}, the morphism $s$ is injective by Theorem \ref{thm: injection}. The morphism $s$ clearly respects the action of $H(\mathrm{B}\mathbb{C}^*,\BQ)=\mathbb{Q}[u]$ defined in \eqref{actionofBGM} and thus 
$$
s(\tbpsw[u]) = \hbar \tbps[u].
$$
By the cohomological integrality Theorem \ref{thm:nilpotentintegrality}, $\tbpsw[u]$ generates $\kharat$. \end{proof}

\medskip 

\begin{lemma}
The morphism \[ s: \kharate \rightarrow \khrate \] respects the perverse filtration. Furthermore, if $T$ satisfies the geometric assumption \eqref{eqn:assumption geometric}, then $s$ is strict, in the sense that \[ s(P^d(\kharate))= \mathrm{Im}(s) \cap P^d(\khrate).\]
\end{lemma}

\medskip 

\begin{proof}
The perverse filtrations on $\khrat$ and $\kharat$ are defined by 
\begin{align*} 
&P^d(\khrat) := H(\mathcal{M}^T(\tilde{Q}), \ptau^{\leq d} (\JH^T_{*} \varphi_{\Tr(\tilde{W})} \mathbb{Q}^{\vir})) \\
&P^d(\kharat) := H(\mathcal{M}^{T,\omega\text{-nilp}}(\tilde{Q}),\oi^{!}_{\omega} \ptau^{\leq d}(\JH^T_{*} \varphi_{\Tr(\tilde{W})} \mathbb{Q}^{\vir})) 
\end{align*} 
The morphism $s$ is defined by taking derived global section of the morphism of sheaves
\[ (\oi_\omega)_* \oi_{\omega}^! \JH^T_{*} \varphi_{\Tr(\tilde{W})} \mathbb{Q}^{\vir} \rightarrow \JH^T_{*} \varphi_{\Tr(\tilde{W})} \mathbb{Q}^{\vir}.\] Since $\oi_\omega$ is a closed embedding, $(\oi_\omega)_*$ commutes with the perverse truncation functor and thus we have a natural morphism
 \[ (\oi_\omega)_{*} \ptau^{\leq d}(\oi_{\omega}^! \JH^T_{*} \varphi_{\Tr(\tilde{W})} \mathbb{Q}^{\vir}) \rightarrow \ptau^{\leq d} \JH^T_{*} \varphi_{\Tr(\tilde{W})} \mathbb{Q}^{\vir}.\]
 Taking derived global sections implies that $s$ respects the filtrations. To show that the filtration is strict, we can follow the same argument as in \cite[Proposition 3.18]{jindalnegutbps}. The associated graded algebras with respect to the perverse filtration are $\mathrm{Sym}(\tbpsw[u])$ and $\mathrm{Sym}(\tbps[u])$ respectively, and the morphism $s$ is induced by 
 $$
 \tbpsw[u] \rightarrow \tbps[u].
 $$
 The morphism displayed above is injective since $T$ satisfies the geometric assumption \eqref{eqn:assumption geometric}. Therefore, the induced morphism on the associated graded algebras is also injective, which is equivalent to $s$ being strict. \end{proof}

\medskip

\subsection{BPS Lie algebras via explicit polynomials}
\label{sub:bps explicit}

We will now recall the results of \cite{jindalnegutbps}, which gave an explicit description for the image of
\begin{equation}
\label{eqn:inclusion}
\tbps \hookrightarrow \CV_{\BQ}^+ = \bigoplus_{\bn \in \nn} \ring_{\BQ}[z_{i1},\dots,z_{in_i}]^{\sym}_{i \in I}
\end{equation}
in terms of degree conditions on polynomials

\medskip 

\begin{theorem}
\label{thm:bps}

The image of $\tbps$ under \eqref{eqn:inclusion} is the set
$$
\Big\{ E(z_{i1},\dots,z_{in_i})_{i \in I} \in \ring_{\BQ}[z_{i1},\dots,z_{in_i}]^{\emph{sym}}_{i \in I} \Big\}
$$ 
where the polynomial $E$ in $\bn = (n_i)_{i \in I}$ variables is required to satisfy the following:

\medskip 

\begin{itemize}[leftmargin=*]

\item for any partition $\bn = \bn^1 + \dots + \bn^k$ into non-zero elements of $\nn$, the polynomial
\begin{equation}
\label{eqn:add y}
E(y_1+z_{i,1},\dots,y_1+z_{i,n_i^1}, \dots,y_k+z_{i,n_i-n_i^k+1}, \dots, y_k+z_{i,n_i})_{i \in I}
\end{equation}
satisfies the following degree bound (in $y_1,\dots,y_k$)
\begin{equation}
\label{eqn:degree bound}
\emph{deg}_{y_1,\dots,y_k}(E) \leq \frac {1-k}2 - \sum_{1\leq a < b \leq k} (\bn^a,\bn^b)^{\prime}
\end{equation}

\medskip 

\item for any $I$-partition $P$, the polynomial
$$
E \left(x_{ia},x_{ia}+ \hbar,\dots,x_{ia} + (n_i^{(a)}-1)\hbar\right)_{i \in I, a \in \{1,\dots,d_i\}}
$$
is divisible (over $\ring_{\BQ}$) by
\begin{equation}
\label{eqn:factorfullbps}
\hbar^{|\bn|-1} \prod_{(\alpha : i \rightarrow j) \in \edge} \mathop{\prod_{1 \leq a \leq d_i}}_{1 \leq b \leq d_j} \prod_{c \in \BZ} \left(x_{jb}-x_{ia}+c\hbar - u_\alpha\right)^{\chi_{n_i^{(a)},n_j^{(b)}}(c)}
\end{equation}

\end{itemize}

\end{theorem}

\medskip

\begin{proof} In \cite{jindalnegutbps}, we showed that 
$$
P^1(\mathcal{A}^T_{\tilde{Q},\mathrm{rat}}) \subset \mathcal{A}^T_{\tilde{Q},\mathrm{rat}} \cong \CV^+_{\mathbb{Q}} = \bigoplus_{\bn \in \nn} \ring_{\mathbb{Q}}[x_{i1},\dots,x_{in_i}]_{i \in I}^{\sym}
$$
coincides with the set of polynomials satisfying the condition in the first bullet of the theorem. Similarly, the condition in the second bullet is equivalent to $\hbar E \in \CS^+_{\BQ}$. Moreover, we showed that the morphism $\iota^{\prime}$ of \eqref{eqn:shuffle morphism} strictly preserves the perverse filtration, and thus so does $\iota^{\prime} \circ s$. We conclude that 
\[ 
\iota^{\prime}\circ s(\tbpsw) = P^1(\mathcal{A}^T_{\tilde{Q},\mathrm{rat}}) \cap \iota(\kharat) 
\] 
which gives
\[ 
\iota^{\prime}(\hbar \tbps) = P^1(\mathcal{A}^T_{\tilde{Q}, \mathrm{rat}}) \cap \iota(\kharat).
\] 
Since multiplication by $\hbar$ action doesn't change the perverse degree, $E \in \iota^{\prime}(\mathfrak{g}^T_{\tilde{Q},\tilde{W}})$ iff $E \in P^1$ and $\hbar E \in \mathcal{S}^+_{\BQ}$. Thus the claim follows. \end{proof}

\begin{corollary}

The image of \[ \iota^{\prime}: \khrate \rightarrow \mathcal{V}^+_{\mathbb{Q}}\] is generated by 
\begin{equation} \label{eqn:special2}
e_{\bn,g}^{\prime} = g(z_{i1},\dots,z_{in_i})_{i \in I} \cdot \frac 1{\hbar} \prod_{i \in I} \prod_{1 \leq a ,  b \leq n_i} \left(z_{ib}-z_{ia}+\hbar\right)
\end{equation}
as $\bn=(n_i)_{i \in I}$ ranges over $\mathbb{N}^I$ and $g$ ranges over $R_{\mathbb{Q}}[z_{i1},\cdots, z_{i,n_i}]^{\mathrm{sym}}_{i \in I}$
\end{corollary}

\medskip

\begin{proof}
By the cohomological integrality Theorem \ref{thm:nilpotentintegrality}, it suffices to show that $\iota^{\prime}(\tbps)$ is contained inside the algebra generated by $e_{\bn,g}^{\prime}$. However, we have shown that $\iota^{\prime}(\hbar \fg^{T}_{\tilde{Q},\tilde{W}})$ is inside the algebra generated by $e_{\bn,g}$. Thus it follows that $\iota^{\prime}(\fg^{T}_{\tilde{Q},\tilde{W}})$ is contained inside algebra generated by $e_{\bn,g}/\hbar$, which is precisely $e^{\prime}_{\bn,g}$. \end{proof}

\medskip 

\begin{remark}
In the case of cyclic quivers, Theorem \ref{thm:bps} implies Proposition 17.7 of \cite{J}, which was instrumental to the computation of the equivariant CoHA of preprojective algebra for cyclic quivers. 
\end{remark}

\medskip

\subsection{Kac polynomials}
\label{sub:kac}

We will now consider the graded Lie algebras
\begin{align*}
&\tbps = \bigoplus_{\bn \in \nn \backslash \b0} \tbpsn \\
&\tbpsw = \bigoplus_{\bn \in \nn \backslash \b0} \tbpswn
\end{align*}
where each $\tbpsn$ and $\tbpswn$ also carries a cohomological grading. We write
$$
\text{grdim}_{\BQ}(\tbpsn) = \sum_{d \in \mathbb{Z}} (-1)^d \dim_{\BQ} \left(H^d(\tbpsn)\right)t^{\frac d2}
$$
for the Hilbert polynomial with respect to the cohomological grading, and similarly for the loop-nilpotent version and other cohomologically graded vector spaces. The maps
$$
\tbpsw \rightarrow \tbps \rightarrow \CV^+_{\BQ}
$$
respect the cohomological grading, which appears on the polynomial rings making up $\CV^+_{\BQ}$ via the formulas $\deg z_{ia} = \deg \hbar = \deg u_\alpha = 2$. Therefore, Theorem \ref{thm:bps} implies 
\begin{equation}
\label{eqn:bps}
t^{(\bn,\bn)^{\prime}} \text{grdim}_{\BQ} (\tbpsn) = \text{grdim}_{\BQ} \left\{\begin{split} &\text{polynomials }E(z_{i1},\dots,z_{in_i})_{i \in I} \text{ which} \\ &\text{satisfy the bullets in Theorem \ref{thm:bps}} \end{split} \right\}
\end{equation}
However, we recall that \cite{davison2022integrality} proved the equality \footnote{Formula \eqref{eqn:kac classic} is proved in \loccit for trivial $T$, but \cite[Theorem A, Section 9.2]{davison2022integrality} implies that $\tbps$ is a free $H^T(\cdot)$-module. Thus we have an isomorphism 
\[ 
\tbps \simeq \bps \otimes H^T(\cdot)
\] 
which respects the cohomological grading. As $\text{grdim}_{\BQ}(H^T(\cdot)) = (1-t)^{-r}$, we infer \eqref{eqn:kac classic}.}
\begin{equation}
\label{eqn:kac classic}
\text{grdim}_{\BQ} (\tbpsn) = (1-t)^{-r}A_{Q,\bn}(t^{-1})
\end{equation}
where $A_{Q,\bn}(q)$ denotes the Kac polynomial which counts absolutely indecomposable $\bn$-dimensional representations of the quiver $Q$ over $\BF_q$. Combining the above two displays yields Corollary \ref{cor:kac} (note that the extra power of $t$ in the statement of the Corollary is due to the loop-nilpotency condition therein, see Proposition \ref{prop:bps}).

\bigskip 

\section{ADE and Jordan quivers}
\label{sec:examples}

\medskip 

\noindent For explicit calculations in the CoHA, it is useful to twist its multiplication by a sign, as follows. Given dimension vectors $\bn,\bn' \in \nn$, let 
\begin{equation}
\label{eqn:coha product}
m_{\bn,\bn'} : \CA^T_{\tQ,\tW,\bn} \otimes \CA^T_{\tQ,\tW,\bn'} \rightarrow \CA^T_{\tQ,\tW,\bn+\bn'}
\end{equation}
be the restriction of the CoHA multiplication to the $\bn,\bn'$ component. Let $\mathcal{A}^{T,\chi}_{\tilde{Q},\tilde{W}}$ denote the sign twisted CoHA, where we use the sign-twisted multiplication 
\begin{equation}
\label{eqn:twisted coha product}
m_{\bn,\bn'}^{\chi} := (-1)^{(\bn,\bn')} m_{\bn,\bn'}
\end{equation}
instead of \eqref{eqn:coha product}. We define the sign-twisted loop nilpotent CoHA $\mathcal{A}^{T, \omega\text{-nilp},\chi}_{\tilde{Q},\tilde{W}}$ analogously, and the rational analogues of all these algebras using the subscript ``rat''.

\medskip 

\subsection{Finite type quivers}
\label{sub:ade quiver}

Let $Q$ be an ADE type quiver with vertex set $I$, and let $Y_{\hbar}(\fg)$ be the corresponding Yangian. The latter is the $\mathbb{Q}[\hbar]$-algebra generated by symbols $X^{\pm}_{i,r}$ and $H_{i,r}$, where $i \in I$ and $r \geq 0$, and relations that we will not recall in full (see \cite{CWendl} for reference). Let $Y^+_{\hbar}(\fg)$ denote the subalgebra generated by the $X^+_{i,r}$. We will consider the (loop-nilpotent) CoHA with respect to the torus
$$
T=\mathbb{C}^*
$$
with $\hbar$ being the standard generator of $\BQ[\text{Lie}(T)]$ and
$$
u_\alpha = \frac {\hbar}2
$$
for every arrow $\alpha$ of $Q$. The following is proved in \cite[Theorem D]{Yang_2018}.

\medskip 

\begin{theorem}
The assignment 
$$
X^+_{i,r} \mapsto z^r_{i1} \in \mathcal{A}^{T,\chi}_{\tilde{Q},\tilde{W}, \bs^i, \mathrm{rat}}
$$
(for all $i \in I, r \geq 0$) induces an isomorphism of $\mathbb{Q}[\hbar]$-algebras
\begin{equation}
\label{eqn:phi}
\Phi: Y^+_{\hbar}(\fg) \rightarrow \mathcal{A}^{T,\chi}_{\tilde{Q},\tilde{W},\mathrm{rat}}
\end{equation}
\end{theorem}

\medskip 

\subsubsection{The coproduct} Let $\Delta$ denote the coproduct of $Y_{\hbar}(\fg)$ introduced by Drinfeld. Let us give precise formulas for its action on some particularly important elements of the Yangian (see \cite{Lev2} for reference). Let 
\[
T_{i,1}:= H_{i,1}- \frac{\hbar}{2} H^2_{i,0} .
\] 
The defining relations of the Yangian imply that 
\[ 
[T_{i,1},X^+_{j,r}]=c_{ij} X^+_{j,r+1}.
\] 
where $c_{ij}$ denote the entries of the Cartan matrix. Since the Cartan matrix of $\fg$ is invertible, there are unique rational numbers $\{\lambda_i\}_{i \in I}$ such that the element 
\begin{equation}
\label{eqn:t}
T = \sum_{i \in I} \lambda_i T_{i,1}
\end{equation}
satisfies the identities 
\begin{equation}
\label{eqn:ad t}
[T,X^+_{i,r}]=X^+_{i+1,r}
\end{equation}
for all $i\in I$ and $r \geq 0$. Then the coproduct of $Y_{\hbar}(\fg)$ satisfies
\begin{equation}
\label{eqn:cop ti}
\Delta(T_{i,1})= T_{i,1} \otimes 1 + 1 \otimes T_{i,1} + \hbar  R_i
\end{equation}
for all $i \in I$, where 
\begin{equation}
\label{eqn:r matrix}
R_i = - \sum_{\alpha \in \Delta+} (\alpha,\alpha_i) X^-_{\alpha,0} \otimes X^+_{\alpha,0}. 
\end{equation}
Above, $X_{\alpha,0}^\pm$ are root vectors of $\fg$, and they satisfy the coproduct formulas 
\begin{equation}
\label{eqn:cop x}
\Delta(X^{\pm}_{\alpha,0}) = X^{\pm}_{\alpha,0} \otimes 1+ 1 \otimes X^{\pm}_{\alpha,0} 
\end{equation}
Let $R = \sum_{i \in I} \lambda_i R_i$. We conclude that
\begin{equation}
\label{eqn:cop t}
\Delta(T) = T \otimes 1+ 1 \otimes T + \hbar R.
\end{equation}
The counit satisfies the properties $\epsilon(T_{i,1})=\epsilon(X^{\pm}_{\alpha,0})=0$ for all $i \in I$ and $\alpha \in \Delta^+$. 

\medskip

\subsubsection{Drinfeld-Gavarini dual}Let $\mathbf{Y}_{\hbar}(\fg)$ denote the Drinfeld-Gavarini dual of the Yangian defined in \cite{Gavarini}. It is defined as follows. For any $n \geq 1$, let $\Delta^{(n-1)}: A \rightarrow A^{\otimes n}$ denote the $n-1$th iterated coproduct and let 
\[
\delta_n = (\mathrm{id}-\epsilon)^{\otimes n} \Delta^{(n-1)}: A \rightarrow  A^{\otimes n}
\]
The Drinfeld-Gavarini dual $\mathbf{Y}_{\hbar}(\fg)$ is defined as the $\mathbb{Q}[\hbar]$-subalgebra of $Y_{\hbar}(\fg)$ spanned by elements $z \in Y_{\hbar}(\fg)$ such that 
\begin{equation}
\label{eqn:dg}
\delta_n(z) \in \hbar^n Y_{\hbar}(\fg)^{\otimes n}
\end{equation}
for all $n \geq 1$. It is studied in Appendix A of \cite{FT} by Finkelberg-Tsymbaliuk-Weekes and by Wendlandt in \cite[Section 5]{CWendl}. Let $\mathbf{Y}^+_{\hbar}(\fg)$ denote its positive half, consisting of elements $z \in Y^+_{\hbar}(\fg)$ which satisfy \eqref{eqn:dg}. 

\medskip 

\begin{proposition}
\label{prop:pbw}

The elements
\begin{equation}
\label{eqn:elements}
\mathrm{ad}_T^r(\hbar X^+_{\alpha,0})
\end{equation}
are inside $\mathbf{Y}^+_{\hbar}(\fg)$, for all $\alpha \in \Delta^+$ and $r \geq 0$.

\end{proposition}

\medskip 

\begin{proof} It was shown in \cite[Lemma A.3]{FT} that if $a,b \in \mathbf{Y}_{\hbar}(\fg)$, then 
$$
\frac{[a,b]}{\hbar} \in \mathbf{Y}_{\hbar}(\fg)
$$
Thus it suffices to check that $\hbar X^+_{\alpha,0} \in \mathbf{Y}^+_{\hbar}(\fg)$ and $\hbar T \in \mathbf{Y}_{\hbar}(\fg)$. For the first of these, we use the fact that 
$$
\delta_n(X^+_{\alpha,0}) = \begin{cases} X^+_{\alpha,0} &\text{if }n = 1 \\ 0 &\text{if }n > 1 \end{cases}
$$ 
which implies $\hbar X^+_{\alpha,0} \in \mathbf{Y}^+_{\hbar}(\fg)$. The analogous property for $T$  follows from
$$
\delta_n(T)= \begin{cases} T &\text{if }n=1 \\ (\mathrm{id}-\epsilon)^{\otimes 2}(T \otimes 1 + 1 \otimes T + \hbar R) = \hbar (\mathrm{id}-\epsilon)^{\otimes 2}(R) &\text{if }n = 2 \\ 0&\text{if }n > 2 \end{cases}
$$
(the $n>2$ statement follows from the fact that if we apply the coproduct any positive number of times to \eqref{eqn:cop t}, every summand will contain at least one tensor factor of $1$ due to the explicit form of the $R$-matrix in \eqref{eqn:r matrix}, and $(\text{id} - \epsilon)(1) = 0$). \end{proof}

\medskip 

\begin{theorem}
\label{thm:ade}

The morphism $\Phi$ restricts to an isomorphism of $\mathbb{Q}[\hbar]$-algebras 
\begin{equation}
\label{eqn:ade}
\mathbf{Y}^+_{\hbar}(\fg) \simeq \mathcal{A}^{T,\omega\emph{-nilp},\chi}_{\tilde{Q},\tilde{W},\emph{rat}} 
\end{equation}

\end{theorem}

\medskip 

\begin{proof} Let us recall the map
$$ 
\mathcal{A}^{T,\omega\text{-nilp},\chi}_{\tilde{Q},\tilde{W},\text{rat}}  \xrightarrow{u \cdot} \mathcal{A}^{T,\omega\text{-nilp},\chi}_{\tilde{Q},\tilde{W},\text{rat}}
$$ 
defined in \eqref{eqn:what is u}. It is proved in \cite[Proposition 4.1]{DW} that $u \cdot$ is a derivation. Note that the homomorphism $\Phi$ of \eqref{eqn:phi} intertwines $u \cdot $ with $\text{ad}_T$ of \eqref{eqn:t}, i.e. 
\begin{equation} 
\label{eqn:ucommute}
  \Phi \left([T,z]\right) = u \cdot \left(\Phi(z)\right), \qquad \forall z \in Y_{\hbar}^+(\fg)  
\end{equation}
To see this, it is enough to prove \eqref{eqn:ucommute} for $z = X_{i,r}^+$, in which case it follows immediately from \eqref{eqn:ad t}. It is proven in \cite{davison2022bps} that 
\begin{equation}
\label{eqn:iso lie}
\tbps \simeq \fn^+ \otimes_{\BQ} \mathbb{Q}[\hbar]
\end{equation}
where $\mathfrak{n}^+$ denotes the positive nilpotent subalgebra of $\fg$. For any $i \in I$, the element 
$$
e_{\alpha_i} = 1 \in \tbps \cap \mathcal{A}^{T,\chi}_{\tilde{Q},\tilde{W},\bs^i}
$$
corresponds to the $i$-th simple root vector of $\fn^+$ under \eqref{eqn:iso lie}. For any positive root $\alpha \in \Delta^+$, we can use iterated Lie brackets to define
$$
e_{\alpha} \in \tbps
$$
just like the $\alpha$-th root vector of $\fn^+$ is obtained from iterated Lie brackets from the simple root vectors. By the very definition, we then have
$$
\Phi(X^+_{\alpha,0}) = e_{\alpha}
$$
for all positive roots $\alpha \in \Delta^+$. Therefore, formula \eqref{eqn:ucommute} implies that
\begin{equation}
\label{eqn:jump}
\Phi(\mathrm{ad}_T^r(\hbar X^+_{\alpha,0})) = \hbar e_{\alpha} u^r \in \hbar \tbps[u].
\end{equation}
and so we conclude that 
$$
\Phi \left(\bigoplus_{\alpha \in \Delta^+} \bigoplus_{r = 0}^{\infty} \text{ad}_T^r(\hbar X^+_{\alpha,0})\right) = \hbar \tbps[u]
$$ 
By Corollary \ref{cor:nilpotentassubalgebra}, $\hbar \tbps[u]$ generates $\mathcal{A}^{T,\omega\text{-nilp},\chi}_{\tilde{Q},\tilde{W},\mathrm{rat}}$ and thus \begin{equation}
\label{eqn:containment}
\mathcal{A}^{T,\omega\text{-nilp},\chi}_{\tilde{Q},\tilde{W},\mathrm{rat}} \subseteq  \Phi(\mathbf{Y}^+_{\hbar}(\fg)).
\end{equation}
It remains to prove that the above inclusion is in fact an equality, and to do this, we will compare the graded dimensions (on the right-hand side, we consider the $\nn \times \BZ$ grading determined by $\mathrm{deg}(X^+_{i,r}) = (\bs^i, 2r)$ and $\mathrm{deg}(\hbar) = (0,2)$). For the left-hand side, the cohomological integrality Theorem \ref{thm:nilpotentintegrality} implies that
\[ 
\text{grdim}_{\BQ}\Big( \mathcal{A}^{T,\omega\text{-nilp},\chi}_{\tilde{Q},\tilde{W},\mathrm{rat}}  \Big)  = \text{grdim}_{\BQ} \left( \text{Sym} \left( \hbar \tbps[u]\right) \right)
\]
(the notation ``\text{grdim}'' stands for the $\nn \times \BZ$ graded dimension). Concerning the right-hand side of \eqref{eqn:containment}, it follows by the PBW theorem of \cite[Proposition 5.2]{CWendl} that 
\[ 
\text{grdim}_{\BQ}\left(\mathbf{Y}^+_{\hbar}(\fg)\right) =  \text{grdim}_{\BQ} \left( \text{Sym} \left( \hbar \fn^+[u]\right) \right)
\]
Comparing the two displays above implies that the inclusion in \eqref{eqn:containment} is an equality. \end{proof}

\medskip 

\subsection{The Jordan quiver}
\label{sub:jordan quiver}

The initial appearance of the divisibility conditions of Definition \ref{def:shuffle int} was in \cite{Integral}, where the first-named author used them to produce a shuffle realization of the semi-nilpotent commuting variety. The cohomological analogue of the contents of \loccit is the following: when $Q$ is the Jordan quiver, $\tQ$ is the quiver with one vertex and three loops
$$
\alpha, \dalpha, \omega
$$
on which the torus $T = (\BC^*)^2$ acts with equivariant parameters
$$
-u, -\hbar+u, \hbar 
$$
respectively. Suppose that instead of defining the $\omega$-nilpotent CoHA (as we did in the present paper), one defines the analogous CoHA
$$
\CA^{T,\dalpha\text{-nilp}}_{\tQ,\tW}
$$
corresponding to the condition that $\dalpha$ is nilpotent. The dimensional reduction theorem (see Subsection \ref{sub:injection}) yields an isomorphism
$$
\CA^{T,\dalpha\text{-nilp}}_{\tQ,\tW} \cong H^{\mathrm{BM}}_T \Big(\text{semi-nilpotent commuting stack} \Big)
$$
where the stack in the right-hand side parameterizes pairs of commuting square matrices $X,Y$ with $Y$ being nilpotent, modulo simultaneous conjugation. Then by analogy with Proposition \ref{prop:yu zhao}, we have an isomorphism
$$
 H^{\mathrm{BM}}_T\Big(\text{semi-nilpotent commuting stack} \Big) \cong \left( E \in \bigoplus_{n = 0}^{\infty} \ring[z_{1},\dots,z_{n}]^{\text{sym}} \text{ satisfying \eqref{eqn:specialization jordan}} \right)
$$
where for any partition $n = (n^{(1)} \geq \dots \geq n^{(d)})$ we impose the condition that
\begin{equation}
\label{eqn:specialization jordan}
E \left(x_{a},x_{a}+ u-\hbar,\dots,x_{a} + (n^{(a)}-1)(u-\hbar) \right)_{a \in \{1,\dots,d\}} 
\end{equation}
$$
\text{is divisible by } (u-\hbar)^{n} \left(  \prod_{a=1}^{d} n^{(a)}! \right)  \prod_{1 \leq a,b \leq d} \prod_{c \in \BZ} \left(x_{b}-x_{a}+c(u-\hbar) - u\right)^{\chi_{n^{(a)},n^{(b)}}(c)}
$$
This is readily seen to be the cohomological version of \cite[Theorem 1.4]{Integral}. We note that the semi-nilpotent commuting stack naturally arises in connection to categorified link invariants (specifically the trace of the affine Hecke category) as explained in \cite{GN1, GN2}. 

\medskip 

\begin{remark}

The particular case of Proposition \ref{prop:spec is comm} for the Jordan quiver was recently proved (by different means from ours) in \cite{H}.
    
\end{remark}

\bigskip

\printbibliography

\end{document}